\documentclass[a4paper,oneside,11pt]{article}

\usepackage{amsfonts}
\usepackage{MnSymbol}
\usepackage{nicefrac}
\usepackage{amsthm}
\usepackage{epsfig}
\usepackage[margin = 3cm]{geometry}
 \usepackage{gensymb}

\newcommand{\R}{\mathbb{R}}
\newcommand{\N}{\mathbb{N}}
\def\b#1{\boldsymbol{#1}}
\def\dx{\,\mathrm dx}
\def\ds{\,\mathrm ds}
\def\der{\mathrm{D}}
\def\div{\,\mathrm{div}\,}

\renewcommand{\epsilon}{\varepsilon}

\newcounter{AssCount}
\renewcommand{\theAssCount}{\textbf{(A{\arabic{AssCount}})}}
\newcounter{AssListCount}

\newcommand{\hp}{s}

\theoremstyle{plain}
\newtheorem{theorem}{Theorem}
\newtheorem{lemma}[theorem]{Lemma}

\theoremstyle{remark}
\newtheorem{remark}{Remark}

\begin{document}

\title{Numerical approximation of phase field based shape and topology optimization for fluids 
\thanks{The authors gratefully acknowledge the support of the 
Deutsche Forschungsgemeinschaft via the SPP 1506 entitled ``Transport processes at fluidic interfaces''. 
They also thank  Stephan Schmidt for providing comparison calculations for the rugby example. } 
}

\author{Harald~Garcke\footnotemark[2]
\and Claudia~Hecht\footnotemark[2]
\and Michael~Hinze\footnotemark[3]
\and Christian~Kahle\footnotemark[3]
}
\date{}

\maketitle

\renewcommand{\thefootnote}{\fnsymbol{footnote}}
\footnotetext[2]{Fakult\"at f\"ur Mathematik, Universit\"at Regensburg, 93040 Regensburg, Germany
({\tt \{Harald.Garcke, Claudia.Hecht\}@mathematik.uni-regensburg.de}).}
\footnotetext[3]{Schwerpunkt Optimierung und Approximation, Universit\"at Hamburg, Bundesstrasse 55, 20146 Hamburg, Germany
({\tt \{Michael.Hinze, Christian.Kahle\}@uni-hamburg.de}).}
\renewcommand{\thefootnote}{\arabic{footnote}}

\begin{abstract}
\noindent We consider the problem of finding optimal shapes of fluid domains. The fluid obeys the
Navier--Stokes equations. Inside a holdall container we use a phase field approach using diffuse
interfaces to describe the domain of free flow. We formulate a corresponding optimization problem 
where flow outside the fluid domain is penalized. 
The resulting formulation of the shape optimization problem is shown to be well-posed, 
hence there exists a minimizer, and first order optimality conditions are derived.

For the numerical realization we introduce a mass conserving gradient flow and obtain a
Cahn--Hilliard type system, which is integrated numerically using the finite element method.
An adaptive concept using reliable, residual based error estimation is exploited for the resolution
of the spatial mesh.

The overall concept is numerically investigated and comparison values are provided.
\end{abstract}



\pagestyle{myheadings}
\markboth{H. GARCKE, C. HECHT, M. HINZE, C. KAHLE}{SHAPE OPTIMIZATION FOR NAVIER--STOKES FLOW}

\section{Introduction}
Shape and topology optimization in fluid mechanics is an important mathematical field attracting 
more and more attention in recent years. 
One reason therefore is certainly the wide application fields spanning from optimization 
of transport vehicles like airplanes and cars, over biomechanical  and industrial production
processes to the optimization of music instruments. Due to the complexity of the emerging problems 
those questions have to be treated carefully with regard to modelling, simulation and interpretation of the results. 
Most approaches towards shape optimization, in particular in the field of shape optimization in fluid mechanics, 
deal mainly with numerical methods, 
or concentrate on combining reliable CFD methods to shape optimization strategies 
like the use of shape sensitivity analysis. Anyhow, it is a well-known fact that well-posedness 
of problems in optimal shape design is a difficult matter where only a few analytical results are available 
so far, see for instance 
\cite{ulbrichulbrich,BrandenburgLindemannUlbrichUlbrich_advancedNumMeth_DesignNSflow,
kawohl2000optimal, pironneau,Schmidt_shape_derivative_NavierStokes, sverak}.
In particular, classical formulations of shape optimization problems lack in general existence
of a minimizer and hence the correct mathematical description has to be reconsidered. 
Among first approaches towards well-posed formulations in this field we mention 
in particular the work  \cite{borrvall}, where a porous medium approach is introduced 
in order to obtain a well-posed problem at least for the special case of minimizing 
the total potential power in a Stokes flow. As discussed in \cite{evgrafov,evgrafov2} 
it is not to be expected that this formulation can be extended without further ado to the stationary 
Navier--Stokes equations or to the use of different objective functionals.

In this work we propose a well-posed formulation for shape optimization in fluids, 
which will turn out to even allow for topological changes. 
Therefore, we combine the porous medium approach of \cite{borrvall} and a phase field approach 
including a regularization by the Ginzburg-Landau energy. This results in a diffuse interface problem, 
which can be shown to approximate a sharp interface problem for shape optimization in fluids 
that is penalized by a perimeter term. Perimeter penalization in shape optimization problems was already
introduced by \cite{ambrosioButtazzo} and has since then been applied to a lot of problems 
in shape optimization, see for instance \cite{bourdin_chambolle}. 
Also phase field approximations for the perimeter penalized problems have been discussed in this field, 
and we refer here for instance to \cite{relatingphasefield, bourdin_chambolle, burger}. 
But to the best of our  knowledge, neither a perimeter penalization nor a phase field approach 
has been applied to a fluid dynamical setting before.

Here we  use the stationary incompressible Navier--Stokes equations as a fluid model, 
but we briefly describe how the Stokes equations could also be used here. 
The resulting diffuse interface problem is shown to inherit a minimizer, 
in contrast to most formulations in shape optimization. 
The resulting formulation turns out to be an optimal control problem with control in the coefficients, 
and hence one can derive optimality conditions in form of a variational inequality. 
Thus, we can formulate a gradient flow for the corresponding reduced objective functional 
and arrive in a Cahn--Hilliard type system. 
Similar to \cite{HintermuellerHinzeTber__AFEM_for_CH}, we use a Moreau--Yosida relaxation in order 
to handle the pointwise constraints on the design variable. 
We formulate the finite element discretization of the resulting problem using a 
splitting approach for the Cahn--Hilliard equation. 
The Navier--Stokes system is discretized with the help of Taylor--Hood elements and 
both variables in the Cahn--Hilliard equation are discretized with continuous, 
piecewise linear elements. In addition, we introduce an adaptive concept 
using residual based error estimates and a D\"orfler marking strategy, 
see also \cite{HintermuellerHinzeKahle_AFEM_for_CHNS, HintermuellerHinzeTber__AFEM_for_CH}.

The proposed approach is validated by means of several numerical examples. 
The first one shows in particular that even topological changes are allowed 
during the optimization process. 
The second example is the classical example of optimizing the shape of a ball in an outer flow. 
We obtain comparable results as in the literature and  discuss the results 
for different Reynolds numbers and penalization parameters. 
For this example, comparison values for further investigations are provided. 
As a third example and outlook, we briefly discuss the optimal embouchure of a bassoon, 
which was already examined by an engineering group at the 
Technical University of Dresden, see \cite{grundmann}. 
Besides, the behaviour of the different parameters of the model and their influence 
on the obtained solution in the above-mentioned numerical examples are investigated.

\section{Shape and topology optimization for Navier--Stokes flow}
\label{sec:Analysis} 
We  study the optimization of some objective functional depending on the shape, geometry and topology of a region which is filled with an incompressible Navier--Stokes fluid.
We use a holdall container $\Omega\subset\R^d$ which is fixed throughout this work and fulfills

\begin{list}{\theAssCount}{\usecounter{AssCount}}
  \item\label{a:Omega} $\Omega\subseteq\R^d$, $d\in\{2,3\}$, is a bounded Lipschitz domain with
  outer unit normal $\b n$ such that $\R^d\setminus\overline\Omega$ is connected.
  \setcounter{AssListCount}{\value{AssCount}}
\end{list}
Requiring the complement of $\overline\Omega$ to be connected simplifies certain aspects in the 
analysis of the Navier--Stokes system but could also be dropped, cf. \cite[Remark 2.7]{hecht}. 
As we do not want to prescribe the topology or geometric properties of the optimal fluid region in advance,
we state the optimization problem in the general framework of Caccioppoli sets.
Thus, a set is admissible if it is a measurable subset of $\Omega$ with finite perimeter.
Additionally, we impose a volume constraint by introducing a constant $\beta\in\left(-1,1\right)$
and optimize over the sets with volume equal to $0.5(\beta+1)\left|\Omega\right|$.
Since an optimization problem in this setting lacks in general existence of minimizers,
see for instance \cite{haberjog}, we introduce moreover a perimeter regularization.
Thus the perimeter term, multiplied by some weighting parameter $\gamma>0$ and a constant $c_0=\frac\pi2$
arising due to technical reasons, is added to the objective functional that we want to minimize.
The latter is given by $\int_\Omega f\left(x,\b u,\der\b u\right)\dx$,
where $\b u\in\b U:=\{\b u\in\b H^1(\Omega)\mid\div\b u=0,\b u|_{\partial\Omega}=\b g\}$
denotes the velocity of the fluid, and we assume

\begin{list}{\theAssCount}{\usecounter{AssCount}}\setcounter{AssCount}{\value{AssListCount}}
  \item\label{a:ObjectiveFctl}
  the functional $f:\Omega\times\R^d\times\R^{d\times d}\to\R$ is given such that
  \begin{align*}
    &F:\b H^1(\Omega)\to\R,\\
    &F\left(\b u\right):=\int_\Omega f\left(x,\b u(x),\der\b u(x)\right)\dx
  \end{align*}
  is continuous, weakly lower semicontinuous, radially unbounded in $\b U$, which means
  \begin{align}\label{a:FctlRadiallyUnbounded}
    \lim_{k\to\infty}\left\|\b u_k\right\|_{\b H^1(\Omega)}=+\infty\implies \lim_{k\to\infty}F\left(\b u_k\right)=+\infty
  \end{align}
  for any sequence $\left(\b u_k\right)_{k\in\N}\subseteq\b U$. Additionally, $F|_{\b U}$ has to be bounded from below.
  \setcounter{AssListCount}{\value{AssCount}}
\end{list}
Here and in the following we use the following function space:
$$\b V:=\left\{\b v\in\b H^1_0(\Omega)\mid\div\b u=0\right\}.$$
Additionally, we denote for some
$\varphi\in BV\left(\Omega,\left\{\pm1\right\}\right)$
the set
$E^\varphi:=\{\varphi\equiv 1\}$
and introduce
\begin{align*}
  \b U^\varphi:=\left\{\b u\in\b U\mid\b u=\b 0 \text{ a.e. in }\Omega\setminus
  E^\varphi\right\},
  \quad
  \b V^\varphi:=\left\{\b v\in\b V\mid\b v=\b 0 \text{ a.e. in }\Omega\setminus E^\varphi\right\},
\end{align*}
where we remark, that we denote $\R^d$-valued functions and function spaces of vector valued functions by boldface letters.
\begin{remark}
For the continuity of $F:\b H^1(\Omega)\to\R$, required in Assumption~\ref{a:ObjectiveFctl},
it is sufficient, that $f:\Omega\times\R^d\times\R^{d\times d}\to\R$ 
is a Carath\'{e}odory function, i.e. $f$ fulfills 
for a.e. $x\in\Omega$ a growth condition of the form
\begin{align*}
  \left|f\left(x,\b v,\b A\right)\right|\leq a(x)+b_1(x)|\b v|^p+b_2(x)|\b A|^2,
  \quad\forall \b v\in\R^d,  \b A\in\R^{d\times d}
\end{align*}
for some $a\in L^1(\Omega)$, $b_1,b_2\in L^\infty(\Omega)$ 
and some $p\geq 2$ for $d=2$ and $2\leq p\leq\nicefrac{2d}{d-2}$ for $d=3$.
\end{remark}
\bigskip

For the fluid mechanics, we use Dirichlet boundary conditions on $\partial\Omega$, thus there may be some inflow or some outflow,
and we allow additionally external body forces on the whole domain $\Omega$.

\begin{list}{\theAssCount}{\usecounter{AssCount}}\setcounter{AssCount}{\value{AssListCount} }
  \item\label{a:Forces}
  Here, $\b f\in\b L^2(\Omega)$ is the applied body force and
  $\b g\in\b H^{\frac12}\left(\partial\Omega\right)$  is some given boundary function such that
  $\int_{\partial\Omega}\b g\cdot\b n\ds=0$,
  \setcounter{AssListCount}{\value{AssCount}}
\end{list}
which are assumed to be given and fixed throughout this paper.

A typical objective functional used in this context is the total potential power, which is given by
\begin{align}\label{e:TotalPotPower}
  f\left(x,\b u,\der\b u\right):=\frac\mu2\left|\der\b u\right|^2-\b f(x)\cdot\b u.
\end{align}
In particular, we remark that this functional fulfills Assumption \ref{a:ObjectiveFctl}.

To formulate the problem, we introduce an one-to-one correspondence of Caccioppoli sets and
functions of finite perimeter by identifying $E\subset\Omega$ with
$\varphi:=2\chi_E-1\in BV\left(\Omega,\left\{\pm1\right\}\right)$
and notice that for any
$\varphi\in BV\left(\Omega,\left\{\pm1\right\}\right)$
the set $E^\varphi:=\left\{\varphi=1\right\}$
is the corresponding Caccioppoli set describing the fluid region.
We shall write $P_\Omega(E)$ for the perimeter of $E\subseteq\Omega$ in $\Omega$.
For a more detailed introduction to the theory of Caccioppoli sets and functions of bounded variations
we refer for instance to \cite{evans_gariepy,giusti}.

Altogether we arrive in the following optimization problem:
\begin{align}\label{e:ObjFctlSharp}
  \min_{\left(\varphi,\b u\right)} J_0\left(\varphi,\b u\right)
  :=\int_\Omega f\left(x,\b u,\der\b u\right)\dx+\gamma c_0P_\Omega\left(E^\varphi\right)
\end{align}
subject to
\begin{align*}
  \varphi\in\Phi_{ad}^0
  :=\left\{\varphi\in BV\left(\Omega,\left\{\pm1\right\}\right)\mid
  \int_\Omega\varphi\dx=\beta\left|\Omega\right|, \b U^\varphi\neq\emptyset\right\}
\end{align*}
and
\begin{subequations}
\label{e:StatNSSharpStrong}
\begin{align}
  -\mu\Delta\b u+\left(\b u\cdot\nabla\right)\b u+\nabla p&=\b f &&\text{in }E^\varphi,\label{e:StatNSSharpStrong1}\\
  -\div\b u&=0&&\text{in }\Omega,\\
  \b u&=\b0&&\text{in }\Omega\setminus E^\varphi,\\
  \b u&=\b g&&\text{on }\partial\Omega.
\end{align}
\end{subequations}

We point out that the velocity of the fluid is not only defined on the fluid region $E^\varphi$,
for some $\varphi\in\Phi_{ad}^0$, but rather on the whole of $\Omega$,
where in $E^\varphi$ it is determined by the stationary Navier-Stokes equations,
and on the remainder we set it equal to zero.
And so for an arbitrary function $\varphi\in BV(\Omega,\left\{\pm1\right\})$
the condition $\b u=\b 0$ a.e. in $\Omega\setminus E^\varphi$
and the non-homogeneous boundary data $\b u=\b g$ on $\partial\Omega$ may be inconsistent.
To exclude this case we impose the condition $\b U^\varphi\neq\emptyset$ on the admissible design functions in $\Phi_{ad}^0$. The state constraints \eqref{e:StatNSSharpStrong} have to be fulfilled in the
following weak sense:
find $\b u\in\b U^\varphi$ such that it holds
\begin{align*}
  \int_\Omega\mu\nabla\b u\cdot\nabla\b v+\left(\b u\cdot\nabla\right)\b u\cdot\b v\dx=\int_\Omega\b f\cdot\b v\dx\quad\forall\b v\in\b V^\varphi.
\end{align*}
Even though this shape and topology optimization problem gives rise to a large class of possible solutions,
numerics and analysis prefer more regularity for handling optimization problems.
One common approach towards more analytic problem formulations is a phase field formulation.
It is a well-known fact, see for instance \cite{modica}, that a multiple of the perimeter functional
is the $L^1(\Omega)$-$\Gamma$-limit for $\epsilon\searrow0$ of the Ginzburg-Landau energy,
which is defined by
\begin{align*}
  \mathcal E_\epsilon\left(\varphi\right):=
  \begin{cases}
    \int_\Omega\frac{\epsilon}{2}\left|\nabla\varphi\right|^2+\frac1\epsilon\psi\left(\varphi\right)\dx, & \text{if }\varphi\in H^1(\Omega),\\
    +\infty, & \text{otherwise.}
  \end{cases}
\end{align*}
Here $\psi:\R\to\overline\R$ is a potential with two global minima and in this work we focus on a double obstacle potential given by
\begin{align*}
  \psi(\varphi):=
  \begin{cases}
    \psi_0\left(\varphi\right), & \text{if }\left|\varphi\right|\leq1,\\
    +\infty,&\text{otherwise,}
  \end{cases}\quad \psi_0\left(\varphi\right):=\frac12\left(1-\varphi^2\right).
\end{align*}
Thus replacing the perimeter functional by the Ginzburg-Landau energy in the objective functional,
we arrive in a so-called diffuse interface approximation,
where the hypersurface between fluid and non-fluid region is replaced by a interfacial layer
with thickness proportional to some small parameter $\epsilon>0$.
Then the design variable $\varphi$ is allowed to have values in $\left[-1,1\right]$ instead of only $\pm1$.
To make sense of the state equations in this setting,
we introduce an interpolation function
$\alpha_\epsilon:\left[-1,1\right]\to \left[0,\overline\alpha_\epsilon\right]$
fulfilling the following assumptions:

\begin{list}{\theAssCount}{\usecounter{AssCount}}\setcounter{AssCount}{\value{AssListCount}}
  \item \label{a:Alpha}
  Let $\alpha_\epsilon:\left[-1,1\right]\to\left[0,\overline\alpha_\epsilon\right]$
  be a decreasing, surjective and twice continuously differentiable function for $\epsilon>0$.
  
  It is required that $\overline\alpha_\epsilon>0$ is chosen such that
  $\lim_{\epsilon\searrow0}\overline\alpha_\epsilon=+\infty$ and $\alpha_\epsilon$
  converges pointwise to some function $\alpha_0:[-1,1]\to[0,+\infty]$. Additionally, we impose $\alpha_\delta(x)\geq\alpha_\epsilon(x)$
  if $\delta\leq\epsilon$ for all $x\in\left[-1,1\right]$, $\lim_{\epsilon\searrow0}\alpha_\epsilon(0)<\infty$ and a growth condition
  of the form $\overline\alpha_\epsilon=\hbox{o}\left(\epsilon^{-\frac23}\right)$.%
  
  \setcounter{AssListCount}{\value{AssCount}}
\end{list}
\begin{remark}\label{r:ConvergenceRateTwoDim}
We remark, that for space dimension $d=2$ we can even choose
$\overline\alpha_\epsilon=\hbox{o}\left(\epsilon^{-\kappa}\right)$
for any $\kappa\in(0,1)$.
\end{remark}

By adding the term $\alpha_\epsilon(\varphi)\b u$ to \eqref{e:StatNSSharpStrong1}
we find that the state equations \eqref{e:StatNSSharpStrong}
then ``interpolate'' between the steady-state Navier--Stokes equations
in $\left\{\varphi=1\right\}$ and some Darcy flow through porous medium
with permeability $\overline\alpha_\epsilon^{-1}$ at $\left\{\varphi=-1\right\}$.
Thus simultaneously to introducing a diffuse interface approximation,
we weaken the condition of non-permeability through the non-fluid region.
This porous medium approach has been introduced for topology optimization in fluid flow by \cite{borrvall}.
To ensure that the velocity vanishes outside the fluid region in
the limit $\epsilon\searrow0$ we add moreover a penalization term to
the objective functional and finally arrive in the following phase field formulation of the problem:

\begin{equation}\label{e:ObjFctl}
  \begin{split}
    \min_{\left(\varphi,\b u\right)}J_\epsilon\left(\varphi,\b u\right)&:=
    \int_\Omega\frac12\alpha_\epsilon\left(\varphi\right)\left|\b u\right|^2\dx
    +\int_\Omega f\left(x,\b u,\der\b u\right)\dx\\
    &+\frac{\gamma\epsilon}{2}\int_\Omega\left|\nabla\varphi\right|^2\dx
    +\frac\gamma\epsilon\int_\Omega\psi\left(\varphi\right)\dx
  \end{split}
\end{equation}
subject to
\begin{align}
  \varphi\in\Phi_{ad}
  :=\left\{
  \varphi\in H^1(\Omega)
  \mid\left|\varphi\right|\leq1\text{ a.e. in }\Omega,
  \int_\Omega\varphi\dx=\beta\left|\Omega\right|
  \right\},
  \label{eq:ObjtFctl_PhiConstraint}
\end{align}
and
\begin{subequations}\label{e:StatNSStrong}
\begin{align}
  \alpha_\epsilon(\varphi)\b u-\mu\Delta\b u+\left(\b u\cdot\nabla\right)\b u+\nabla p&=\b f &&\text{in }\Omega,\\
  -\div\b u&=0&&\text{in }\Omega,\\
  \b u&=\b g&&\text{on }\partial\Omega.
\end{align}
\end{subequations}

Considering the state equations \eqref{e:StatNSStrong}, we find the following solvability result:
\begin{lemma}\label{l:StateEquationsWellDefined}
For every $\varphi\in L^1(\Omega)$ such that $\left|\varphi\right|\leq1$ a.e. in $\Omega$
there exists some $\b u\in\b U$ such that \eqref{e:StatNSStrong} is fulfilled in the following sense:
\begin{align}\label{e:StatNSWeak}
  \int_\Omega\alpha_\epsilon\left(\varphi\right)\b u\cdot\b v
  +\mu\nabla\b u\cdot\nabla\b v
  +\left(\b u\cdot\nabla\right)\b u\cdot\b v\dx
  =\int_\Omega\b f\cdot\b v\dx\quad\forall\b v\in\b V.
\end{align}
Besides, if there exists a solution $\b u\in\b U$ of \eqref{e:StatNSWeak} such that it holds
\begin{align}\label{e:SmallnessForUnique}
  \left\|\nabla\b u\right\|_{\b L^2(\Omega)}<\frac{\mu}{K_\Omega},
  \quad
  K_\Omega:=
  \begin{cases}
    \nicefrac23\sqrt2|\Omega|^{\frac23}, & \text{if }d=3,\\
    0.5\sqrt{\left|\Omega\right|}, & \text{if }d=2,
  \end{cases}
\end{align}
then this is the only solution of \eqref{e:StatNSWeak}.
\end{lemma}
\begin{proof}
The existence proof is based on the theory on pseudo-monotone operators and the uniqueness statement follows similar
to classical results concerning stationary 
Navier--Stokes equations, see for instance \cite{galdi,hecht}.
\end{proof}
\begin{remark}\label{r:PressureStateEquations}
	Standard results infer from \eqref{e:StatNSWeak} that there exists a pressure 
	$p\in L^2(\Omega)$ associated to $\b u\in\b U$ such that \eqref{e:StatNSStrong} 
	is fulfilled in a weak sense, see \cite{galdi}. 
	But as we are not considering the pressure dependency in the optimization problem, 
	we drop those considerations in the following. For details on how to include the 
	pressure in the objective functional in this setting we refer to \cite{hecht}. 
\end{remark}

Using this result, one can show well-posedness of the optimal control problem in the phase field formulation
stated above by exploiting the direct method in the calculus of variations.
\begin{theorem}
There exists at least one minimizer $(\varphi_\epsilon,\b u_\epsilon)$ of
\eqref{e:ObjFctl}--\eqref{e:StatNSStrong}.
\end{theorem}

The proof is given in \cite{hecht}.

\bigskip

To derive first order necessary optimality conditions for a solution 
$(\varphi_\epsilon,\b u_\epsilon)$ of \eqref{e:ObjFctl}--\eqref{e:StatNSStrong} 
we introduce the Lagrangian 
$\mathcal L_\epsilon:\Phi_{ad}\times\b U\times\b V\to\R$ by
\begin{align*}
  \mathcal L_\epsilon\left(\varphi,\b u,\b q\right):=
  J_\epsilon(\varphi,\b u)
  -\int_\Omega\alpha_\epsilon\left(\varphi\right)\b u\cdot\b q+\mu\nabla\b u\cdot\nabla\b q+\left(\b u\cdot\nabla\right)\b u\cdot\b q-\b f\cdot\b q\dx.
\end{align*}
The variational inequality is formally derived by
\begin{equation}
  \begin{split}\label{e:LagrangianVariationalInequality}
    \mathrm D_\varphi\mathcal L_\epsilon\left(\varphi_\epsilon,\b u_\epsilon,\b
    q_\epsilon\right)\left(\varphi-\varphi_\epsilon\right),
    %
    %
		\geq0\quad\forall\varphi\in\Phi_{ad}
  \end{split}
\end{equation}
and the adjoint equation can be deduced by
\begin{align*}
  \mathrm D_{\b u}\mathcal L_\epsilon\left(\varphi_\epsilon,\b u_\epsilon,\b q_\epsilon\right)\left(\b v\right)
  %
  %
=0\quad\forall\b v\in\b V.
\end{align*}

Even though those calculations are only formally, we obtain therefrom  a first order optimality
system, which can be proved to be fulfilled for a minimizer of the optimal control problem stated
above, see \cite{hecht}:

\begin{theorem}\label{t:OptimalitySystem}
Assume $\left(\varphi_\epsilon,\b u_\epsilon\right)\in\Phi_{ad}\times\b U$
is a minimizer of \eqref{e:ObjFctl}--\eqref{e:StatNSStrong} such that
$\left\|\nabla\b u_\epsilon\right\|_{\b L^2(\Omega)}<\nicefrac{\mu}{K_\Omega}$.
Then the following variational inequality is fulfilled:
\begin{equation}\label{e:VariationalInequality}
  \begin{split}
    \left(\frac12\alpha'_\epsilon\left(\varphi_\epsilon\right)\left|\b u_\epsilon\right|^2
    +\frac\gamma\epsilon\psi'_0\left(\varphi_\epsilon\right)
    -\alpha'_\epsilon\left(\varphi_\epsilon\right)\b u_\epsilon\cdot\b q_\epsilon
    +\lambda_\epsilon,
    \varphi-\varphi_\epsilon\right)_{L^2(\Omega)}\\
    +\left(\gamma\epsilon\nabla\varphi_\epsilon,\nabla\left(\varphi-\varphi_\epsilon\right)\right)
    _{\b L^2(\Omega)}\geq0\quad\forall\varphi\in\overline{\Phi}_{ad},
  \end{split}
\end{equation}
with
\begin{align*}
  \overline{\Phi}_{ad}
  :=\left\{\varphi\in H^1(\Omega)\mid\left|\varphi\right|\leq1\,\text{ a.e. in  }\Omega\right\},
\end{align*}
where $\b q_\epsilon\in\b V$ is the unique weak solution to the following adjoint system:
\begin{subequations}\label{e:AdjointSystem}
\begin{align}
  \alpha_\epsilon\left(\varphi_\epsilon\right)\b q_\epsilon
  -\mu\Delta\b q_\epsilon
  +\left(\nabla\b u_\epsilon\right)^T\b q_\epsilon
  -\left(\b u_\epsilon\cdot\nabla\right)\b q_\epsilon
  +\nabla\pi_\epsilon&=\alpha_\epsilon\left(\varphi_\epsilon\right)\b u_\epsilon\nonumber\\
  &\hspace{-4cm}+\der_2f\left(\cdot,\b u_\epsilon,\der\b u_\epsilon\right)-\div \der_3f\left(\cdot,\b u_\epsilon,\der\b u_\epsilon\right)&&\text{in }\Omega,\\
  -\div\b q_\epsilon&=0&&\text{in }\Omega,\\
  \b q_\epsilon&=\b0&&\text{on }\partial\Omega.
\end{align}
\end{subequations}
Here, we denote by $\der_if\left(\cdot,\b u_\epsilon,\der\b u_\epsilon\right)$ with $i=2$ and $i=3$ the differential of $f:\Omega\times\R^d\times\R^{d\times d}$ with respect to the second and third component, respectively. Besides, $\b u_\epsilon$ solves the state equations $\eqref{e:StatNSStrong}$ corresponding to
$\varphi_\epsilon$ in the weak sense and $\lambda_\epsilon\in\R$ is a Lagrange multiplier for the
integral constraint. Additionally, $\pi_\epsilon\in L^2(\Omega)$ can as in Remark~\ref{r:PressureStateEquations} be obtained as pressure associated to the adjoint system.
\end{theorem}

Under certain assumptions on the objective functional it can be verified that a minimizer $(\varphi_\epsilon,\b u_\epsilon)$ of
\eqref{e:ObjFctl}--\eqref{e:StatNSStrong}
fulfills $\left\|\nabla\b u_\epsilon\right\|_{\b L^2(\Omega)}<\nicefrac{\mu}{K_\Omega}$. This implies by Lemma~\ref{l:StateEquationsWellDefined} that $\b u_\epsilon$ is the only solution of \eqref{e:StatNSStrong} corresponding to $\varphi_\epsilon$, see \cite{hecht}. In particular, for minimizing the total potential power,
see \eqref{e:TotalPotPower},
this condition is equivalent to stating ``smallness of data or high viscosity''
as can be found in classical literature.
For details and the proof of Theorem \ref{t:OptimalitySystem}
we refer the reader to \cite{hecht}.

Hence it is not too restrictive to assume from now on that in a neighborhood of the minimizer $\varphi_\epsilon$ 
the state equations \eqref{e:StatNSStrong} are uniquely solvable, 
such that we can introduce the reduced cost functional $j_\epsilon(\varphi):=J_\epsilon(\varphi,\b u)$ where $\b u$ 
is the solution to \eqref{e:StatNSStrong} corresponding to $\varphi$. 
The optimization problem \eqref{e:ObjFctl}--\eqref{e:StatNSStrong} is then equivalent to 
$\min_{\varphi\in\Phi_{ad}}j_\epsilon(\varphi)$.

Following \cite{HintermuellerHinzeTber__AFEM_for_CH}, 
we consider a Moreau--Yosida relaxation 
of this optimization problem 
\begin{align}\tag{$\hat P_\infty$}
  \min_{\varphi\in \Phi_{ad}} j_\epsilon(\varphi)
\end{align}
in which the primitive constraints $|\varphi| \leq 1$ $a.e.$ in $\Omega$ are replaced (relaxed)
through an additional quadratic penalization term in the cost functional.
The optimization problem then reads
\begin{align}\tag{$\hat P_\hp$}
  \min_{\varphi\in H^1(\Omega), \int_\Omega\varphi\dx=\beta|\Omega|} j_\epsilon^s(\varphi),
\end{align}
where
\begin{align}
  \label{e:ObjFctlRelaxed}
  j_\epsilon^s(\varphi)
  :=j_\epsilon(\varphi)+\frac s2\int_\Omega\left|\max\left(0,\varphi-1\right)\right|^2\dx
  +\frac s2\int_\Omega\left|\min\left(0,\varphi+1\right)\right|^2\dx.
\end{align}
Here, $s \gg 1$ plays the role of the penalization parameter. 
The associated Lagrangian $\mathcal L_\epsilon^s$ reads then correspondingly
\begin{equation}
  \begin{split}
    \mathcal L_\epsilon^s\left(\varphi,\b u,\b q\right)&:=
    J_\epsilon(\varphi,\b u)+\frac s2\int_\Omega\left|\max\left(0,\varphi-1\right)\right|^2\dx+\frac s2\int_\Omega\left|\min\left(0,\varphi+1\right)\right|^2\dx\\
    &-\int_\Omega\alpha_\epsilon\left(\varphi\right)\b u\cdot\b q+\mu\nabla\b u\cdot\nabla\b q+\left(\b u\cdot\nabla\right)\b u\cdot\b q-\b f\cdot\b q\dx.
  \end{split}
\end{equation}
Similar analysis as above yields the gradient equation
\begin{equation}\label{e:VariationalEqualityS}
  \begin{split}
     &\mathrm D_\varphi\mathcal L_\epsilon^s\left(\varphi_\epsilon,\b u_\epsilon,\b q_\epsilon\right)\varphi
     =\left(\frac12\alpha'_\epsilon\left(\varphi_\epsilon\right)\left|\b u_\epsilon\right|^2+\frac\gamma\epsilon\psi'_0\left(\varphi_\epsilon\right)-\alpha'_\epsilon\left(\varphi_\epsilon\right)\b u_\epsilon\cdot\b q_\epsilon+\lambda_s(\varphi_\epsilon),\varphi\right)_{L^2(\Omega)}\\
		&+\left(\gamma\epsilon\nabla\varphi_\epsilon,\nabla\varphi\right)_{\b L^2(\Omega)}=0,
  \end{split}
\end{equation}
which has to hold for all $\varphi\in H^1(\Omega)$ with 
$\int_\Omega\varphi\dx=0$. 
Here we use
$\lambda_s(\varphi_\epsilon)=\lambda_s^+(\varphi_\epsilon)+\lambda_s^-(\varphi_\epsilon)$ 
with $\lambda_s^+(\varphi_\epsilon):=s\max\left(0,\varphi_\epsilon-1\right)$ 
and $\lambda_s^-(\varphi_\epsilon):=s\min\left(0,\varphi_\epsilon+1\right)$, 
and $\b q_\epsilon\in\b V$ is the adjoint state given as weak solution of \eqref{e:AdjointSystem}. 
The functions $\lambda_s^+(\varphi_\epsilon)$ and $\lambda_s^-(\varphi_\epsilon)$ 
can also be interpreted as approximations of  Lagrange multipliers for the pointwise constraints
$\varphi\leq1$ a.e.
in $\Omega$ and $\varphi\geq-1$ a.e. in $\Omega$, respectively.

It can be shown, that the sequence of minimizers
$\left(\varphi_\epsilon,\b u_\epsilon\right)_{\epsilon>0}$
of \eqref{e:ObjFctl}--\eqref{e:StatNSStrong} has a subsequence that converges in
$L^1(\Omega)\times\b H^1(\Omega)$ as $\epsilon\searrow 0$.
If the sequence $\left(\varphi_\epsilon\right)_{\epsilon>0}$
converges of order $\mathcal O\left(\epsilon\right)$ one obtains
that the limit element actually is a minimizer of
\eqref{e:ObjFctlSharp}--\eqref{e:StatNSSharpStrong}.
In these particular cases, one can additionally prove that the first order optimality conditions given by
Theorem \ref{t:OptimalitySystem} are an approximation of the classical shape derivatives
for the shape optimization problem \eqref{e:ObjFctlSharp}--\eqref{e:StatNSSharpStrong}.
For details we refer the reader to \cite{hecht}.

\begin{remark}
The same analysis and considerations can be carried out in a Stokes flow.
For the typical example of minimizing the total potential power
\eqref{e:TotalPotPower} it can then even be shown,
that the reduced objective functional corresponding to the phase field formulation
$\Gamma$-converges in $L^1(\Omega)$ to the reduced objective functional
of the sharp interface formulation.
Moreover, the first order optimality conditions are much  simpler
since no adjoint system is necessary any more.
For details we refer to \cite{hecht}.
\end{remark}

\section{Numerical solution techniques}
To solve the phase field problem \eqref{e:ObjFctl}--\eqref{e:StatNSStrong} numerically, we use a
steepest descent approach. For this purpose, we assume as above that in a neighborhood
of the minimizer $\varphi_\epsilon$ the state equations \eqref{e:StatNSStrong} are uniquely solvable,
and hence the reduced cost functional
$j_\epsilon(\varphi):=J_\epsilon(\varphi,\b u)$, with $\b u$  the solution to
\eqref{e:StatNSStrong} corresponding to $\varphi$, is well-defined. 
In addition, we introduce an artificial time variable $t$.
Our aim consists in finding a stationary point in $\Phi_{ad}$ of the following gradient flow:
\begin{align}\label{e:Gradientflow}
  \left\langle\partial_t\varphi,\zeta\right\rangle
  =-\mathrm{grad} j_\epsilon^s(\varphi)(\zeta)
  =-\mathrm Dj_\epsilon^s(\varphi)(\zeta)\quad\forall\zeta\in H^1(\Omega), \int_\Omega\zeta\dx=0,
\end{align}
with some inner product $\left\langle\cdot,\cdot\right\rangle$, 
where $j_\epsilon^s$ is the Moreau--Yosida relaxed cost functional defined in \eqref{e:ObjFctlRelaxed}. 
This flow then decreases the cost functional $j_\epsilon^s$.\\
Now a stationary point $\varphi_\epsilon\in\Phi_{ad}$ of this flow fulfills
the necessary optimality condition \eqref{e:VariationalEqualityS}.
Obviously, the resulting equation depends on the choice of the inner product.
Here, we choose an $H^{-1}$-inner product which is defined as
\begin{align*}
  \left(v_1,v_2\right)_{H^{-1}(\Omega)}:=
  \int_\Omega\nabla\left(-\Delta\right)^{-1}v_1\cdot\nabla\left(-\Delta\right)^{-1}v_2\dx,
\end{align*}
where $y=(-\Delta)^{-1}v$ for $v\in \left(H^1(\Omega)\right)^\star$ with $\left< v,1\right>=0$ is the weak solution of $-\Delta y=v$ in $\Omega$, $\partial_\nu y=0$ on $\partial\Omega$.
The gradient flow \eqref{e:Gradientflow} with this particular choice of
$\left\langle\cdot,\cdot\right\rangle=\left(\cdot,\cdot\right)_{H^{-1}(\Omega)}$
reads as follows:
\begin{align*}
  \partial_t\varphi&=\Delta w&&\quad\text{in }\Omega,\\
  \left(-w,\xi\right)_{L^2(\Omega)}&=-\mathrm Dj_\epsilon^s(\varphi)(\xi) &&\quad\forall\xi\in
  H^1(\Omega), \int_\Omega\xi\dx=0,
\end{align*}
together with homogeneous Neumann boundary conditions on $\partial\Omega$ for $\varphi$ and $w$.
The resulting problem can be considered as a generalised Cahn--Hilliard system.
It follows from direct calculations that this flow preserves the mass,
i.e. $\int_\Omega\varphi(t,x)\dx=\int_\Omega\varphi(0,x)\dx$ for all $t$.
In particular, no Lagrange multiplier for the integral constraint is needed any more.
After fixing some initial condition $\varphi_0\in H^1(\Omega)$ such that
$\left|\varphi_0\right|\leq1$ a.e. and $\int_\Omega\varphi_0\dx=\beta\left|\Omega\right|$,
and some final time $T>0$ this results in the following problem:\\

\noindent\parbox{13cm}{
\noindent\underline{\textit{Cahn--Hilliard System:}}

\medskip

\noindent	Find sufficiently regular $\left(\varphi,w,\b u\right)$
such that
\begin{subequations}\label{eq:MY:CahnHilliard}
\begin{align}
  \partial_t\varphi&=\Delta w&&\text{in }\Omega\times(0,T),\label{eq:MY:CahnHilliard:first}\\
  -\gamma\epsilon\Delta\varphi  +\lambda_\hp(\varphi) + \frac\gamma\epsilon\psi'_0(\varphi)
  +  \alpha'_\epsilon(\varphi)
  \left(
  \frac12 \left|\b u\right|^2-\b u\cdot\b q
  \right)
  &=w &&\text{in }\Omega\times(0,T),\label{eq:MY:CahnHilliard:second} \\
  \varphi(0)&=\varphi_0&&\text{in }\Omega,\\
  \partial_\nu\varphi=0,\, \partial_\nu w&=0 &&\text{on }\partial\Omega\times\left(0,T\right),
\end{align}
\end{subequations}
where $\b u(t)$ fulfills the state equations \eqref{e:StatNSStrong} corresponding to $\varphi(t)$,
and $\b q(t)$ is the adjoint variable defined by \eqref{e:AdjointSystem}.}

\medskip

\subsection{Numerical implementation}
For a numerical realization of the gradient flow method for finding (locally) optimal topologies
we discretize the systems \eqref{e:StatNSStrong}, \eqref{e:AdjointSystem} and
\eqref{eq:MY:CahnHilliard} in time and space.

For this let $0 = t_0< t_1 <\ldots<t_k<t_{k+1}<\ldots$ denote a time grid with
step sizes $\tau_k = t_{k}-t_{k-1}$. For ease of presentation we use a fixed step size and thus set
$\tau_k \equiv \tau$, but we note, that in our numerical implementation $\tau$ is adapted to the
gradient flow in direction $\nabla w$, see Section \ref{ssec:num:TimeStepLength}.

Next a discretization in space  using the finite element method is performed.
For this let $\mathcal{T}^{k}$ denote a conforming triangulation of $\Omega$ with closed simplices
$T\subset \overline \Omega$.
For simplicity we assume that $\overline \Omega$ is exactly represented by $\mathcal{T}^{k}$, 
i.e. $\overline \Omega = \bigcup_{T\in \mathcal{T}^k}T$.
The set of faces of $\mathcal{T}^{k}$ we denote by $\mathcal{E}^{k}$, while the set of nodes
we denote by $\mathcal{N}^{k}$.
For each simplex $T\in \mathcal{T}^k$ 
we denote its diameter by $h_T$, 
and for each face $E\in\mathcal{E}^k$ 
its diameter by $h_E$.
We introduce
the finite element spaces
\begin{align*} 
  \mathcal{V}^1(\mathcal{T}^{k}) &=
  \{v\in C(\overline\Omega)\,|\,v|_T \in P_1(T), \,\forall  T\in \mathcal{T}^{k}\},\\
  \b{\mathcal{V}}^2_{\b g_h}(\mathcal{T}^{k}) &=
  \{v\in C(\overline\Omega)^d\,|\,v|_T \in P_2(T)^d, \,\forall  T\in \mathcal{T}^{k},\,
  v|_{\partial\Omega} = \b g_h\}, 
\end{align*}
where $P_k(T)$ denotes the set of all polynomials up to order $k$ defined on the triangle $T$. The
boundary data $\b v|_{\partial\Omega} = \b g$ is incorporated by a suitable approximation $\b g_h$
of $\b g$ on the  finite element mesh. 

Now at time instance $t_k$ we by
$\b u_h \in \b{\mathcal{V}}^2_{\b g_h}(\mathcal{T}^{k+1})$
denote the fully discrete variant of $\b u$ and by $\b q_h\in\b{\mathcal{V}}^2_0(\mathcal{T}^{k+1})$ 
the fully discrete variant of $\b q$.
Accordingly we proceed with the discrete variants
$\varphi_h, w_h,p_h,\pi_h \in \mathcal{V}^1(\mathcal{T}^{k})$ of $\varphi,w,p$, and $\pi$, where
$\int_\Omega p_h\dx =  \int_\Omega \pi_h\dx = 0$ is required.

Let $\b q^k$ and $\varphi^k$  denote the adjoint velocity and the phase field variable from
the  time step $t_k$, respectively. 
At time instance $t_{k+1}$ we consider
\begin{subequations}\label{eq:FE:NavierStokes}
\begin{align}
  \alpha_\epsilon(\varphi^k)\b u_h - \mu \Delta \b u_h
  + \left(\b u_h \cdot \nabla\right) \b u_h + \nabla p_h &= \b f,\label{eq:FE:NavierStokes1}\\
  \div \b u_h &= 0,\label{eq:FE:NavierStokes2}
\end{align}
\end{subequations}
\begin{subequations}\label{eq:FE:Adjoint}
\begin{align}
  \alpha_\epsilon(\varphi^k)\b q_h- \mu \Delta \b q_h
  - \left(\b u_h \cdot \nabla\right) \b q_h + \nabla \pi_h &=
  \alpha_\epsilon(\varphi^k)\b u_h + \der_2f(\cdot,\b u_h,\der\b u_h)\\
	&-\div\der_3f\left(\cdot,\b u_h,\der\b u_h\right) - \left(\nabla \b  u_h\right)^T\b q^k,\\
  \div \b q_h &= 0, 
\end{align}
\end{subequations}
\begin{subequations}\label{eq:FE:CahnHilliard}
\begin{align}
  \tau^{-1}(\varphi_h -\varphi^k) - \Delta w_h &= 0,\label{eq:FE:CahnHilliard:first}\\
  -\gamma\epsilon \Delta \varphi_h
  + \lambda_\hp(\varphi_h) 
  + \frac{\gamma}{\epsilon}\psi_0'(\varphi^k)
  +\alpha_\epsilon'(\varphi_h)\left(\frac{1}{2}|\b u_h|^2 - \b u_h\cdot \b q_h\right)
  &=  w_h,\label{eq:FE:CahnHilliard:second}
\end{align}
\end{subequations}
as discrete counterpart to \eqref{e:StatNSStrong}, \eqref{e:AdjointSystem} and
\eqref{eq:MY:CahnHilliard}, respectively.
 
The weak form of \eqref{eq:FE:CahnHilliard} using 
$\psi_0'(\varphi^k) = -\varphi^k$ reads
\begin{subequations}\label{eq:FE:CahnHilliard_weak}
\begin{align}
  F^1((\varphi_h,w_h),v) &=\tau^{-1}(\varphi_h -\varphi^k,v)_{L^2(\Omega)} + \left(\nabla w_h,\nabla v\right)_{\b L^2(\Omega)}  =0,
  &\hspace{-1cm} \forall  v\in \mathcal{V}^1(\mathcal{T}^{k+1}),
  \label{eq:FE:CahnHilliard_weak:first}\\
  F^2((\varphi_h,w_h),v) &= \gamma\epsilon (\nabla \varphi_h,\nabla v)_{\b L^2(\Omega)}
  + (\lambda_\hp(\varphi_h),v)_{L^2(\Omega)} - \frac{\gamma}{\epsilon}(\varphi^k,v)_{L^2(\Omega)}\nonumber\\
  &\hspace{-2cm}+\left(\alpha_\epsilon'(\varphi_h)\left(\frac{1}{2}|\b u_h|^2 - \b u_h\cdot \b q_h\right),v\right)_{L^2(\Omega)}
  -(w_h,v)_{L^2(\Omega)} = 0, & \hspace{-1cm}\forall  v \in  \mathcal{V}^1(\mathcal{T}^{k+1}). 
  \label{eq:FE:CahnHilliard_weak:second}
\end{align}
\end{subequations}

The time discretization is chosen to obtain a sequential coupling of the three equations of
interest. Namely to obtain the phase field on time instance $t_{k+1}$ we first solve
\eqref{eq:FE:NavierStokes} for $\b u_h$ using the phase field $\varphi^k$ from the previous time
step. With $\b u_h$ and $\varphi^k$ at hand we then solve \eqref{eq:FE:Adjoint} to obtain the
adjoint velocity $\b q_h$ which then together with $\b u_h$ is used to obtain a new phase field 
$\varphi^{k+1}$ from \eqref{eq:FE:CahnHilliard}.

\begin{remark}
It follows from the structure of \eqref{eq:FE:NavierStokes}--\eqref{eq:FE:CahnHilliard}, that
$\varphi_h$ and $\b u_h,\b q_h$ could be discretized on different spatial grids. In the numerical
part we for simplicity use one grid for all variables involved.
\end{remark}

To justify the discretization \eqref{eq:FE:NavierStokes}--\eqref{eq:FE:CahnHilliard} we state the
following assumptions.
\begin{list}{\theAssCount}{\usecounter{AssCount}}
  \setcounter{AssCount}{\value{AssListCount}}
  \item \label{a:alphaBounded}
  The interpolation function $\alpha_\epsilon:[-1,1] \to [0,\overline {\alpha_\epsilon}]$ is extended
  to $\tilde\alpha_\epsilon:\mathbb{R}\to \mathbb{R}$ fulfilling Assumption \ref{a:Alpha}, so that
  there exists $0\leq\delta<\infty$ such that $\tilde\alpha_\epsilon(\varphi)\geq -\delta$ for all
  $\varphi\in\mathbb{R}$, with $\delta$  sufficiently small. For convenience we in the following
  do not distinguish $\alpha_\epsilon$ and $\tilde\alpha_\epsilon$.
  \setcounter{AssListCount}{\value{AssCount}}
\end{list}
\begin{list}{\theAssCount}{\usecounter{AssCount}}
  \setcounter{AssCount}{\value{AssListCount}}
  \item \label{a:uumuq}
  For given $\varphi^k\in\mathcal{V}^1(\mathcal{T}^k)$ let $\b u_h$ denote the solution to
  \eqref{eq:FE:NavierStokes} and $\b q_h$ denote the corresponding solution to \eqref{eq:FE:Adjoint}.
  Then there holds
  \begin{align*}
    \frac12|\b u_h|^2 - \b u_h\cdot \b q_h \geq 0.
  \end{align*}
  \setcounter{AssListCount}{\value{AssCount}}
\end{list}
\begin{list}{\theAssCount}{\usecounter{AssCount}}\setcounter{AssCount}{\value{AssListCount}}
  \item \label{a:alphaConvex}
 	Additional to Assumption \ref{a:Alpha}, we assume that $\alpha_\epsilon$ is convex. 	 
  \setcounter{AssListCount}{\value{AssCount}}
\end{list}

\begin{remark}
Assumption \ref{a:alphaBounded} is required to ensure existence of unique solutions to
\eqref{eq:FE:NavierStokes} and \eqref{eq:FE:Adjoint} if $\delta$ is sufficiently small.
   
Assumption \ref{a:uumuq} is fulfilled in our numerics for small Reynolds numbers but can not be
justified analytically. 
This assumption might be neglected if $\alpha_\epsilon'$ is discretized explicitly in time in
\eqref{eq:FE:CahnHilliard:second}. 
Due to the large values that  $\alpha_\epsilon'$ takes, we
expect a less robust behaviour of the numerical solution process if we discretize
$\alpha_\epsilon'$ explicitly in time.   
   
Using Assumption \ref{a:alphaBounded} and Assumption \ref{a:uumuq} the existence of a unique
solution to \eqref{eq:FE:CahnHilliard} follows from \cite{HintermuellerHinzeTber__AFEM_for_CH}.

For a general $\alpha_\epsilon$ one can use a splitting $\alpha_\epsilon = \alpha_\epsilon^+ +
\alpha_\epsilon^-$ where $\alpha_\epsilon^+$ denotes the convex part of $\alpha_\epsilon$ and
$\alpha_\epsilon^-$ denotes the concave part. Then $\alpha_\epsilon^+$ is discretized implicitly in
time as in \eqref{eq:FE:CahnHilliard:second}, and $\alpha_\epsilon^-$ is discretized explicitly in
time to obtain a stable discretization, see e.g.
\cite{eyre_CH_semi_implicite,GarckeHinzeKahle_CHNS_AGG_linearStableTimeDisc}.

\end{remark}

The system \eqref{eq:FE:NavierStokes} is solved by an Oseen iteration, where at step $j+1$ of the
iteration the transport $\b u_h^j$ in the nonlinear term $(\b u_h^j \cdot \nabla)\b u_h^{j+1}$ is
kept fix and the resulting linear Oseen equation is solved for $(\b u_h^{j+1},p_h^{j+1})$.
The existence of solutions to the Oseen equations for solving \eqref{eq:FE:NavierStokes} and the
Oseen equation \eqref{eq:FE:Adjoint} are obtained from
\cite[Th. II 1.1]{GiraultRaviart_FEM_for_NavierStokes} using Assumption \ref{a:alphaBounded}.

In \eqref{eq:FE:Adjoint} we use the adjoint variable from the old time instance for
discretizing $\left(\nabla \b u\right)^T \b q$ in time.
In this way \eqref{eq:FE:Adjoint} yields a discretized Oseen equation for which efficient
preconditioning techniques are available.

As mentioned above, the nonlinearity in system \eqref{eq:FE:NavierStokes} is solved by an Oseen
fixed-point iteration. The resulting linear systems are solved by a preconditioned gmres
iteration, see \cite{Saad_gmres}.
The restart is performed depending on the parameter $\mu$ and yields a restart after 10 to 40 iterations.
The employed preconditioner is of upper triangular type, see e.g.
\cite{Benzi_numericalSaddlePoint}, including the $F_p$ preconditioner from
\cite{kayLoghinWelford_FpPreconditioner}. 
The  block arising from the momentum equation \eqref{eq:FE:NavierStokes1} is
inverted using umfpack \cite{UMFPACK}. Since \eqref{eq:FE:Adjoint} is an Oseen equation the same procedure is used for
solving for $\b q_h$. 

The gradient equation \eqref{eq:FE:CahnHilliard} is solved by Newton's method,
see \cite{HintermuellerHinzeTber__AFEM_for_CH} for details in the case of the pure Cahn--Hilliard
equation.
For applying Newton's method to \eqref{eq:FE:CahnHilliard} Assumption \ref{a:uumuq} turns out to be numerically essential.
The linear systems appearing in Newton's method are solved directly using umfpack \cite{UMFPACK}.
Here we also refer to \cite{BoschStollBenner_fastSolutionCH}
concerning iterative solvers and preconditioners for the solution
of the Cahn--Hilliard  equation with Moreau--Yosida relaxation.

The simulation of the gradient flow is stopped as soon as $\|\nabla
w_h\|_{L^2(\Omega)}\leq tol_{abs} + tol_{rel}\|w^0\|_{L^2(\Omega)}$ holds.
Typically we use $tol_{abs} = 10^{-6}$ and $tol_{rel}=10^{-12}$.

\subsubsection{The adaptive concept}\label{ssec:adaptConcept}
For resolving the interface which separates the fluid and the porous material we adapt the adaptive concept
provided in \cite{HintermuellerHinzeKahle_AFEM_for_CHNS,HintermuellerHinzeTber__AFEM_for_CH} to the
present situation.
We base the concept only upon the gradient flow structure, thus the
Cahn--Hilliard equation, and derive a posteriori error estimates up to higher order terms for the
approximation of $\nabla\varphi$ and $\nabla w$.

\medskip

\noindent We define the following errors and residuals:
\begin{align*}
  e_\varphi &= \varphi_h-\varphi,
  & e_w &= w_h-w,\\
  r_h^{(1)} &= \varphi_h-\varphi^k,
  & r_h^{(2)} &=   \alpha_\epsilon^\prime(\varphi_h)
  \left(\frac{1}{2}|\b u_h|^2-\b u_h\cdot\b q_h\right)
  +\lambda_\hp(\varphi_h) - \frac{\gamma}{\epsilon}\varphi^k-w_h,\\
  \eta_{T_E}^{(1)} &= 
  \sum_{E\subset T}\!\!h_E^{1/2}\|\!\left[\!\nabla w_h\!\right]_E\! \|_{\b
  L^2(E)}, & \eta_{T_E}^{(2)} &= 
  \sum_{E\subset T}\!\!h_E^{1/2}\|\!\left[\!\nabla  \varphi_h\!\right]_E\!\|_{\b L^2(E)}, \\
  \eta_N^{(1)} &= h_N^2\|r_h^{(1)}-R_N^{(1)}\|_{L^2(\omega_N)}^2,
  & \eta_N^{(2)} &= h_N^2\|r_h^{(2)}-R_N^{(2)}\|_{L^2(\omega_N)}^2.
\end{align*}
The values  $\eta_N^{(i)},\, i=1,2$ are node-wise error values, while
$\eta_{T_E}^{(i)},\, i=1,2$ are edgewise error contributions, where for each triangle $T$ the
contributions over all edges of $T$ are summed up.
For a node $N \in \mathcal{N}^{k+1}$ we by
$\omega_N$  denote the support of the piecewise linear basis function
located at $N$ and set $h_N := \mbox{diam}(\omega_N)$. 
The value $R_N^{(i)} \in
\mathbb{R},\, i=1,2$ can be chosen arbitrarily. Later they represent appropriate means.
By $[\cdot]_E$ we denote the jump across the face $E$ in normal 
direction $\nu_E$ pointing from simplex with smaller global number to simplex with larger global
number. $\nu_E$ denotes the outer normal at $\Omega$ of $E\subset \partial \Omega$.

To obtain a residual based error estimator we follow the construction in
\cite[Sec. 7.1]{HintermuellerHinzeTber__AFEM_for_CH}. We further use
\cite[Cor. 3.1]{Carstensen_QuasiInterpolation} to obtain lower bounds for the terms $\eta_N^{(1)}$
and $\eta_N^{(2)}$.
For convenience of the reader we state \cite[Cor. 3.1]{Carstensen_QuasiInterpolation} here.
\begin{theorem}[{\cite[Cor. 3.1]{Carstensen_QuasiInterpolation}}]
\label{cor:CarstensenInterpolation}
There exists a constant $C>0$ depending on the domain $\Omega$ and on the regularity of the
triangulation $\mathcal T$ such that
\begin{align*}
  \int_\Omega R(u-\mathcal I u)\,\mathrm dx
  + &\int_\mathcal{E}J(u-\mathcal I u)\ds\\
  &\leq
  C\|\nabla u\|_{\b L^p(\Omega)}
  \left(
  \sum_{N\in \mathcal N}h_N^q\|R-R_N\|_{L^p(\omega_N)}^q
  + \sum_{T\in \mathcal T}h_T\|J\|_{L^q(\mathcal{E}\cap \partial T)}^q
  \right)^{1/q}
\end{align*}
holds for all $J\in L^q(\mathcal E)$, $R\in L^q(\Omega)$, $u\in W^{1,p}(\Omega)$,
and arbitrary $R_N\in \mathbb R$ for $N\in \mathcal N$, 
where  $1<p,q<\infty$
satisfy $\frac1p+\frac1q=1$.
\end{theorem}

Here $\mathcal I:L^1(\Omega) \to  \mathcal{V}^{\mathcal T}$ denotes a modification of
the Cl\'ement interpolation operator proposed in
\cite{Carstensen_QuasiInterpolation,CarstensenVerfuerth_EdgeResidualDominate}.
In \cite{CarstensenVerfuerth_EdgeResidualDominate} it is shown, that in general the error
contributions  arising from the jumps of the gradient of the discrete objects dominate the error
contributions arising from triangle wise residuals.
In our situation it is therefore sufficient to use
the error indicators $\eta_{T_E}^{(i)}$, $i=1,2$, in an adaptation  scheme to obtain well resolved
meshes.
let us assume that $R\in H^1(\Omega)$ in Corollary \ref{cor:CarstensenInterpolation}. 
Then with
$R_N = \int_{\omega_N} R\dx$ we obtain $\|R-R_N\|_{L^2(\Omega)}\leq C(\omega_N)\|\nabla R\|_{\b
L^2(\Omega)}$, and $C(\omega_N)\leq \mbox{diam}(\omega_N)\pi^{-1}$, cf. \cite{Payne_PoincareConstant}.

Since the construction of the estimator is  standard we here only briefly describe the procedure.
We use the errors  $e_w$ and $e_\varphi$ as test functions in 
 \eqref{eq:FE:CahnHilliard_weak:first} and \eqref{eq:FE:CahnHilliard_weak:second}, respectively.
Since $e_w,e_\varphi\in H^1(\Omega)$ they are valid test functions 
in  \eqref{eq:MY:CahnHilliard}. Subtracting 
\eqref{eq:FE:CahnHilliard_weak:first} and the weak form of \eqref{eq:MY:CahnHilliard:first},
tested by $e_w$, as well as subtracting  \eqref{eq:FE:CahnHilliard_weak:second} and the
weak form of \eqref{eq:MY:CahnHilliard:second}, tested by $e_\varphi$ and adding the resulting
equations yields
\begin{align*}
  &\tau \|\nabla e_w\|_{\b L^2(\Omega)}^2 + \gamma\epsilon\|\nabla e_\varphi\|^2_{\b L^2(\Omega)}\\
   &+ \left(\lambda_\hp(\varphi_h)- \lambda_\hp(\varphi),e_\varphi\right)_{L^2(\Omega)}
  + \left( \left[\alpha_\epsilon^\prime(\varphi_h)-\alpha_\epsilon^\prime(\varphi) \right]
  \left(\frac{1}{2}|\b u_h|^2 - \b u_h \cdot \b q_h\right),e_\varphi \right)_{L^2(\Omega)}\\
  &\leq F^{(1)}((\varphi_h,w_h),e_w) + F^{(2)}((\varphi_h,w_h),e_\varphi)\\
  &+\left(
  \alpha_\epsilon^\prime(\varphi)
  \left[
  \left(\frac{1}{2}|\b u|^2 -  \b u \cdot \b q\right) -
  \left(\frac12|\b u_h|^2 -  \b u_h \cdot \b q_h\right)
  \right] 
  ,e_\varphi\right)_{L^2(\Omega)}.
\end{align*} 
For convenience we investigate the term $F^{(1)}((\varphi_h,w_h),e_w)$.
Since $\mathcal I e_w \in \mathcal{V}^1(\mathcal T^{k+1})$ it is a valid
test function for   \eqref{eq:FE:CahnHilliard_weak:first}.
We obtain
\begin{align*}
  F^{(1)}((\varphi_h,w_h),e_w) &= F^{(1)}((\varphi_h,w_h),e_w-\mathcal{I}e_w)\\
  &=\tau^{-1}\int_\Omega (\varphi_h-\varphi^k)(e_w-\mathcal{I}e_w)\dx 
  + \int_\Omega \nabla  w_h\cdot \nabla (e_w-\mathcal{I}e_w)\dx\\
  &+ \tau^{-1}\int_\Omega r_h^{(1)}(e_w-\mathcal{I}e_w)\,\mathrm dx
  +  \sum_{E\subset \mathcal E}\int_E \left[ \nabla w_h \right]_E (e_w-\mathcal{I}e_w)\ds.
\end{align*}
Applying Corollary \ref{cor:CarstensenInterpolation} now gives
\begin{align*}
  F^{(1)}&((\varphi_h,w_h),e_w)\\
  &\leq
  C\|\nabla e_w\|_{\b L^2(\Omega)}
  \left(
  \tau^{-2}\sum_{N\in\mathcal{N}} h_N^2\|r_h^{(1)}\|^2_{L^2(\omega_N)}
  +\sum_{T\in \mathcal T} h_T \|\left[ \nabla w_h \right]_E\|_{L^2(\partial T)}^2
  \right)^{1/2}.
\end{align*}
For $F^{(2)}((\varphi_h,w_h),e_\varphi)$ a similar result holds. Using Young's inequality we obtain
the following theorem.

\begin{theorem}  
There exists a constant $C>0$ independent of $\tau,\gamma,\epsilon,\hp$ and 
$h:=\max_{T\in\mathcal{T}}h_T$ such that there holds:
\begin{align*}
  &\tau \|\nabla e_w\|^2_{\b L^2(\Omega)} + \gamma\epsilon\|\nabla e_\varphi\|^2\\
  & + \left(\lambda_\hp(\varphi_h)- \lambda_\hp(\varphi),e_\varphi\right)_{L^2(\Omega)}
  + \left( \left[\alpha_\epsilon^\prime(\varphi_h)-\alpha_\epsilon^\prime(\varphi) \right]
  (\frac{1}{2}|\b u_h|^2 - \b u_h \cdot \b q_h),e_\varphi \right)_{L^2(\Omega)}\\
  &\leq C\left(\eta_\Omega^2 + \eta_{h.o.t.}^2\right),
\end{align*}
where 
\begin{align*}
  \eta_\Omega^2 :=
  \frac{1}{\tau}\sum_{N\in \mathcal{N}^{k+1}} \left(\eta_N^{(1)}\right)^2
  +\frac{1}{\gamma\epsilon}\sum_{N\in \mathcal{N}^{k+1}} \left(\eta_N^{(2)}\right)^2
  +\tau\sum_{T\in \mathcal{T}^{k+1}} \left(\eta_E^{(1)}\right)^2
  +\gamma\epsilon\sum_{T\in \mathcal{T}^{k+1}} \left(\eta_E^{(2)}\right)^2,
\end{align*}
and 
\begin{align*}
  \eta_{h.o.t.}^2 :=&
  \frac{1}{\gamma\epsilon}\sum_T \left\|\alpha_\epsilon^\prime(\varphi)
  \left( \left( \frac{1}{2}|\b u_h|^2-\b u_h\cdot\b q_h\right) - \left(\frac{1}{2}|\b u|^2 - \b
  u\cdot \b q\right) \right)\right\|_{L^2(T)}^2.
\end{align*}
\end{theorem}

\begin{remark}
\begin{enumerate}
\item Since $\lambda_\hp$ is monotone there holds
$\left(\lambda_\hp(\varphi_h)-\lambda_\hp(\varphi),e_\varphi\right)_{L^2(\Omega)}\geq 0$.
\item We note that due to Assumption \ref{a:uumuq} and
the convexity of $\alpha_\epsilon$ we obtain
$\left( \left[\alpha_\epsilon^\prime(\varphi_h)-\alpha_\epsilon^\prime(\varphi) \right]
(\frac{1}{2}|\b u_h|^2 - \b u_h\cdot \b q_h),e_\varphi \right)_{L^2(\Omega)} \geq 0$.
\item Due to using quadratic elements for both the velocity field $\b u_h$ and the adjoint
velocity field $\b q_h$ we expect that the term $\eta_{h.o.t.}$ can be further estimated with higher
powers of $h$. It therefore is  neglected in our numerical implementation.
\item The values $R_N^{(i)},\,i=1,2$ can be chosen arbitrarily in $\mathbb{R}$. By using
the mean value $R_N^{(i)} = \int_{\omega_N} r_h^{(i)}\,\mathrm dx$ and the Poincar\'e-Friedrichs
inequality together with estimates on the value of its constant (\cite{Payne_PoincareConstant})
the terms $\eta_N^{(i)},\, i=1,2$ are expected to be of higher order and thus are also are neglected
in the numerics.
\item Efficiency of the estimator up to terms of higher order
can be shown along the lines of \cite[Sec. 7.2]{HintermuellerHinzeTber__AFEM_for_CH} by the standard
bubble technique, see e.g. \cite{AinsworthOden_Aposteriori}.
\end{enumerate}
\end{remark}

For the adaptation process we use the
error indicators $\eta_{T_E}^{(1)}$ and $\eta_{T_E}^{(2)}$  in the following  D\"orfler
marking strategy (\cite{Doerfler}) 
as in \cite{HintermuellerHinzeKahle_AFEM_for_CHNS,HintermuellerHinzeTber__AFEM_for_CH}.

\paragraph{The adaptive cycle}
We define the simplex-wise error indicator $\eta_{T_E}$ as
\begin{align*}
  \eta_{T_E} =\eta_{T_E}^{(1)} + \eta_{T_E}^{(2)},
\end{align*}
and the set of admissible simplices
\begin{align*}
  \mathcal{A} = \{ T\in \mathcal{T}^{k+1}\,|\, a_{\min}\leq |T| \leq a_{\max} \},
\end{align*}
where $a_{\min}$ and $a_{\max}$ are the a priori chosen minimal and maximal sizes of simplices.
For adapting the computational mesh we use the following  marking strategy:
\begin{enumerate}
\item Fix constants $\theta^r$ and $\theta^c$ in $(0,1).$
\item Find a set $\mathcal{M}^{E} \subset \mathcal{T}^{k+1}$ such
that
\[
\sum_{T\in\mathcal{M}^{E}} \eta_{T_E} \geq \theta^r
\sum_{T\in\mathcal{T}^{k+1}} \eta_{T_E}.
\]
\item Mark each $T\in (\mathcal{M}^{E} \cap \mathcal{A})$  for refinement.
\item Find the set $\mathcal{C}^E \subset \mathcal{T}^{k+1}$ such
that for each $T \in \mathcal{C}^E$ there holds
\begin{align*}
  \eta_{T_E} & \leq \frac{\theta^c}{N_T}
  \sum_{T\in\mathcal{T}^{k+1}} \eta_{T_E}.
\end{align*}
\item Mark all $T \in \left ( \mathcal{C}^{E}  \cap \mathcal{A}\right )$ for coarsening.
\end{enumerate}
Here  $N_T$ denotes the number of elements of $\mathcal{T}^{k+1}$.

\medskip

We note that by this procedure a simplex can both be marked for refinement and coarsening. In this
case it is refined only. We further note, that we apply this cycle once per time step and then
proceed to the next time instance.


\section{Numerical examples}
In this section we discuss how to choose the values incorporated 
by our porous material -- diffuse interface approach. 

We note that  there are a several approaches on  topology optimization in Navier--Stokes
flow, see e.g.
\cite{borrvall,HansenHaberSigmun_TopoOptChannelFlow,KreisslMaute_fluidTopoOptXFEM,
OlesenOkelsBruus_TopoOptSteadyStateNS,pingen_TopoOpt_Boltzmann}.
On the other hand it seems, that so far no quantitative values to describe the optimal
shapes are available in the literature.
All publications we are aware of  give
qualitative results or quantitative results that seem not to be normalized for comparison with
other codes.

In the following we start with fixing the interpolation function $\alpha_\epsilon$ and the
parameters $\tau$, $\hp$, and $\epsilon$.
We thereafter in Section \ref{ssec:num:TH} investigate how the phase field approach can
find optimal topologies starting from a homogeneously distributed porous material.

In Section \ref{ssec:num:RB} we present numerical experiments for the rugby ball, see als
\cite{borrvall}, \cite{pingen_TopoOpt_Boltzmann} and \cite{Schmidt_shape_derivative_NavierStokes}. 
Here we provide comparison value for the friction drag of the optimized shape, and
as second comparison value we introduce the circularity describing the deviation of the ball from a
circle.

As last example, and as outlook, we address the optimal shape of the embouchure of a bassoon in
Section \ref{ssec:num:FG}.

In the following numerical examples we always assume the absence of external forces,
hence $\b f\equiv\b 0$. 
The optimization aim is always given by minimizing the dissipative energy 
\eqref{e:TotalPotPower},
which in the absence of external forces is given by
\begin{align*}
  F = \int_\Omega \frac{\mu}{2}|\nabla \b u|^2\dx.
\end{align*}
 
The Moreau--Yosida parameter in all our computations is set to $\hp=10^6$. We do not investigate
its couplings to the other parameters involved.
 
For later referencing we here state the parabolic in-/outlet boundary data that we use frequently
throughout  this section
\begin{align}
  \label{eq:num:parabolicBoundary}
  g(x) =
  \begin{cases}
    h\left(1-\left(\frac{x-m}{l/2}\right)^2\right) & \mbox{if } |x-m|<l/2, \\
    0 & \mbox{otherwise}.
  \end{cases}
\end{align}
In the following
this function denotes the normal component of the boundary data at portions of the boundary, where
inhomogeneous Dirichlet boundary conditions are prescribed.
The tangential component is set to zero if not mentioned differently.

\subsection{Time step adaptation}\label{ssec:num:TimeStepLength}
For a faster convergence towards optimal topologies we 
adapt the length of the time steps $\tau^{k+1}$.
Here we use a CFL-like condition to ensure that the interface is not moving too fast into the direction of the flux $\nabla w_h$. 
With
\begin{align*}
  \tau^*= \min_{T\in \mathcal{T}^k} \frac{h_T}{\|\nabla w^k\|_{L^{\infty}(T)}}
\end{align*}
we set
\begin{align*}
  \tau^{k+1} = \max(\tau_{\max},\tau^*),
\end{align*}
where $\tau_{\max}$ denotes an upper bound on the allowed step size and typically is set to
$\tau_{\max} = 10^4$.
Thus the time step size for the current step is calculated using the variable
$w^k$ from the previous time instance. We note that especially for $\nabla w^k \to 0$ we obtain $t_k\to
\infty$, and thus when we approach the final state, we can use arbitrarily large time steps. 
We further note, that if we choose a constant time step the convergence towards a stationary point of the gradient flow in all our
examples is very slow and that indeed large time steps close to the equilibrium are required.

\subsection{The interfacial width}\label{ssec:num:epsilon}
As discussed in Section~\ref{sec:Analysis} the phase field problem
can be verified to approximate the sharp interface shape optimization problem
as $\epsilon\searrow 0$ in a certain sense.
Hence we assume that the phase field problems yield
reasonable approximations of the solution for fixed
but small $\epsilon>0$,
and we do not vary its value.
Typically, in the following we use  the fixed value $\epsilon=0.005$.

\subsection{The interpolation function}\label{ssec:num:alphaepsilon}
We set (see \cite{borrvall})
\begin{align}
  \alpha_\epsilon(\varphi)
  := \frac{\overline \alpha}{2\sqrt{\epsilon}} (1-\varphi)\frac{q}{(\varphi+1+q)},
  \label{eq:num:alpha}
\end{align}
with $\overline  \alpha > 0$ and $q>0$. In our numerics we set $\overline \alpha = 50$. In Figure
\ref{fig:num:alpha} the function $\alpha_\epsilon$ is depicted in dependence of $q$.
We have
$\alpha_\epsilon(-1)=\overline{\alpha}_\epsilon =\overline{\alpha}\epsilon^{-1/2}$ 
and Assumption \ref{a:Alpha} is fulfilled, except that 
$\lim_{\epsilon \searrow 0}\alpha_\epsilon(0) < \infty$ holds. 
Anyhow, the numerical results with this choice of
$\alpha_\epsilon$ are reasonable and we expect that this limit condition has to be posed for
technical reasons only.
\begin{figure}
  \centering
  \epsfig{file=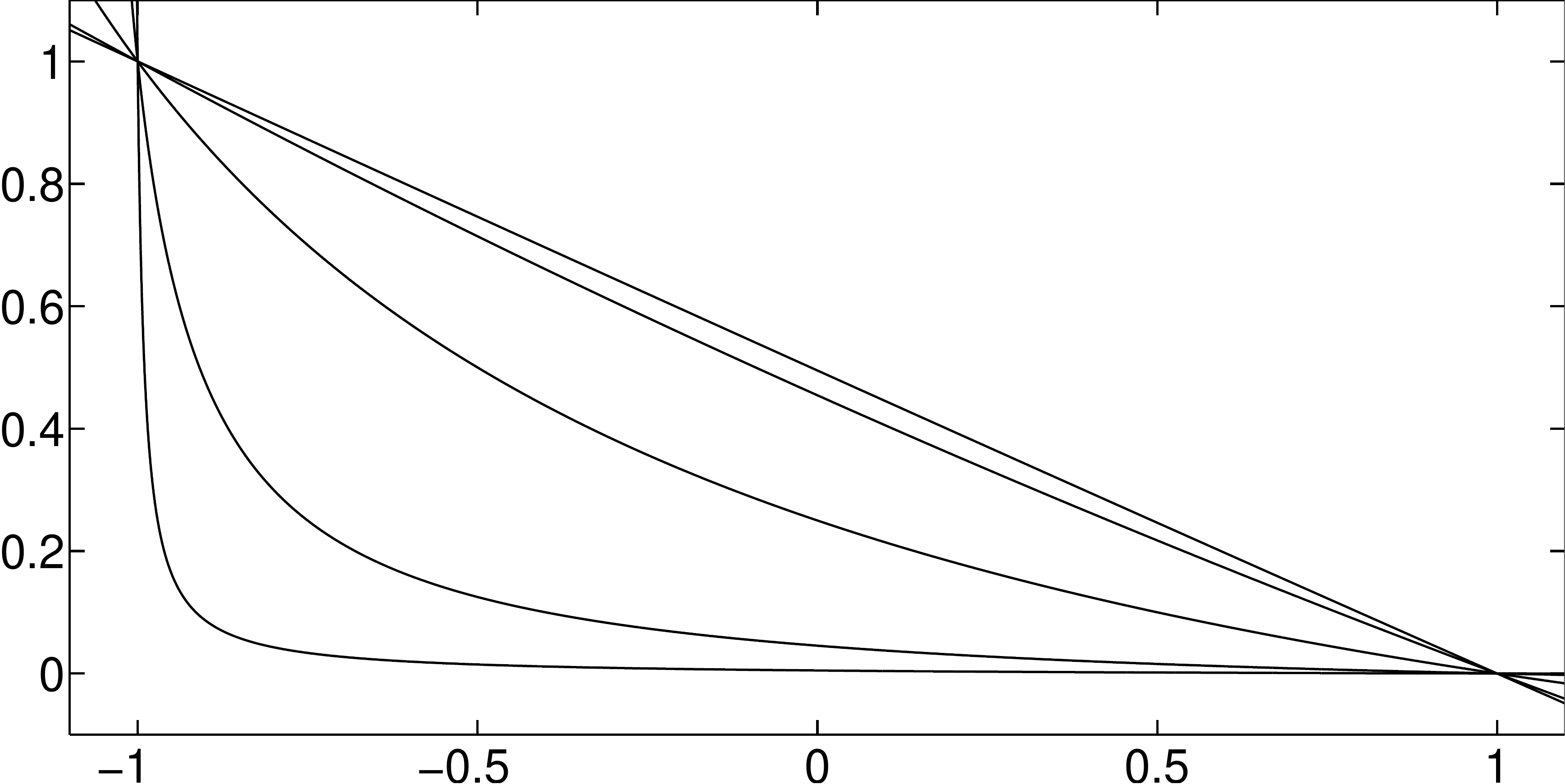, width=6cm,height=3cm}
  \caption{The shape of the  interpolation function $\alpha_\epsilon$ for $q = 10^i, i=-2,\ldots 2$
  (bottom to top).}
  \label{fig:num:alpha}
\end{figure}
To fulfill Assumption~\ref{a:alphaBounded} we  cut
$\alpha_\epsilon$ at  $\varphi\equiv\varphi_c>1$ and use any smooth continuation yielding
$\alpha_\epsilon(\varphi)\equiv \mbox{const}$ for $\varphi\geq \varphi_c$.

The parameter $q$  controls the width of the transition zone between fluid and porous
material.
In \cite{borrvall} the authors typically use a rather small value of $q=0.01$. They also show how
different values of $q$ might lead  to different local optimal topologies.
Since here we also have the parameter $\epsilon$ for controlling the maximal width of the
transition zone we fix $q:=10$.

The fluid material is assumed to be located at $\varphi = 1$ where $\alpha(1)=0$ holds.
Since we use Moreau--Yosida
relaxation we allow $\varphi$ to take values larger then $+1$ and smaller then $-1$.
The choice of $q=10$ and $\hp=10^6$ in our setting always guarantees, that $\varphi+1+q\gg 0$ holds,
and that the violation of $\alpha_\epsilon(\varphi)\geq 0$ at $\varphi=1$ only is small.


Using an interpolation function that yields a smooth transition to zero at $\varphi=1$, say a
polynomial of order 3, in our numerics especially for small values of $\gamma$ yields undesired
behaviour of the numerical solvers.
We for example obtain that fluid regions disappear resulting in a constant porous
material.
The reason is, that then $\alpha_\epsilon(\beta)\approx 0$
if $\beta$ is chosen in the flat region of $\alpha_\epsilon$.
If $\beta$ can be chosen small enough,
the choice of $\varphi\equiv\beta$ yields constant porous material
and hence a very small total potential power.
Thus, $\varphi\equiv\beta$ is at least a local minimizer.

The benefit of small values of $q$ described in \cite{borrvall} stays valid and
for large values of $\gamma$, say $\gamma=1$, small values of $q$ can help finding a valid
topology  when starting from a homogeneous material. This property is the reason
to use this function instead of a linear one,
although $\alpha_\epsilon$ can be regarded as linear for the  value of $q=10$ that we use here.

\subsubsection{The influence of $\overline{\alpha}_\epsilon$ on the interface}
\label{ssec:num:alphaInterface}
The separating effect not only arises from the contribution of the  Ginzburg--Landau energy,
but also the term \linebreak[4]
$\alpha_\epsilon(\varphi)\left(\frac{1}{2}|\b u|^2 - \b u \cdot \b q\right)$ 
yields the demixing of fluid and porous material.
Since $\alpha_\epsilon$ scales with $\overline{\alpha}_\epsilon$ we next
investigate the relative effect of $\overline{\alpha}$ and $\epsilon$
concerning the demixing and thus the width of the resulting interface.
This is done for several values of the parameter $\gamma$, which weights  the two separating forces.

The numerical setup for this test is described in the following.
In the computational domain $\Omega = (0,1)^2$ we have a parabolic inlet at $x\equiv 0$
with $m = 0.5, l=0.2$, and $h=1$. At $x\equiv 1$ we have an outlet with the same values.
The viscosity is set to $\mu = 1$.
We investigate the evolution of the value
\begin{equation*}
  I = \frac{\int_{\{|\varphi|\leq 1\}}\dx}{\int_{\{\varphi=0\}}\ds},
\end{equation*}
which is the size of the area of the transition zone between fluid and material, and which is
normalized by the length of the interface. This value estimates the thickness of the
interfacial region.
For this test we fix $\epsilon\equiv 1$ in \eqref{eq:num:alpha}
and use the interpolation
function 
\begin{equation*}
  \alpha(\varphi) = \alpha_1(\varphi)
  = \frac{\overline \alpha}{2} (1-\varphi)\frac{q}{(\varphi+1+q)},
\end{equation*}
so that $\overline \alpha\equiv \alpha(-1)$.

For fix $\gamma$ we calculate the optimal topology for several combinations of $\epsilon$ and
$\overline{\alpha}$. In Figure \ref{fig:num:atEps} we depict the value of $I$
depending on $\overline{\alpha}$ for several $\epsilon$.
We used $\gamma \in\{0.5,0.05,0.005 \}$ (left to right).

\begin{figure}
  \centering
  \epsfig{file=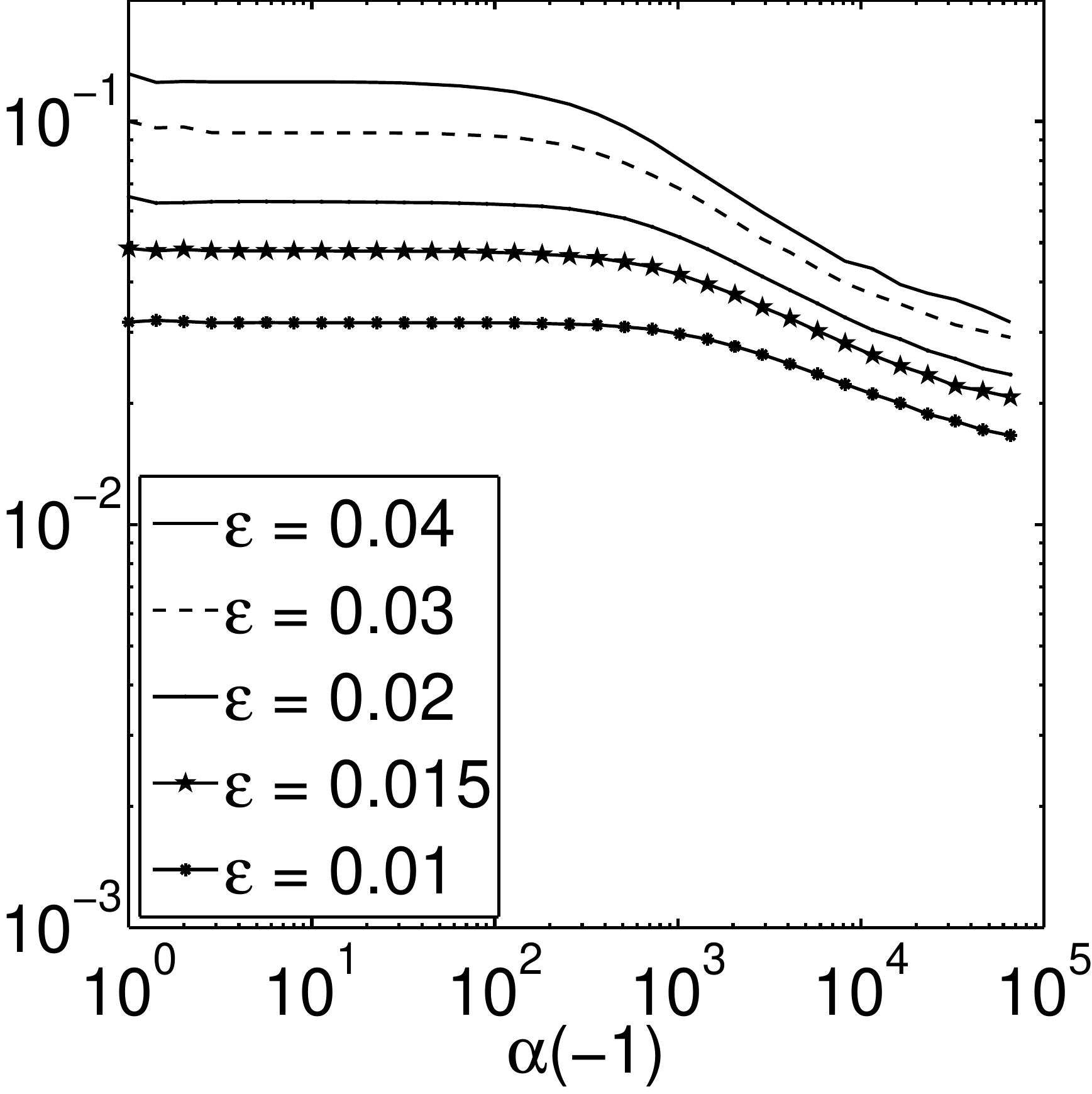,width=0.3\textwidth}\hfill
  \epsfig{file=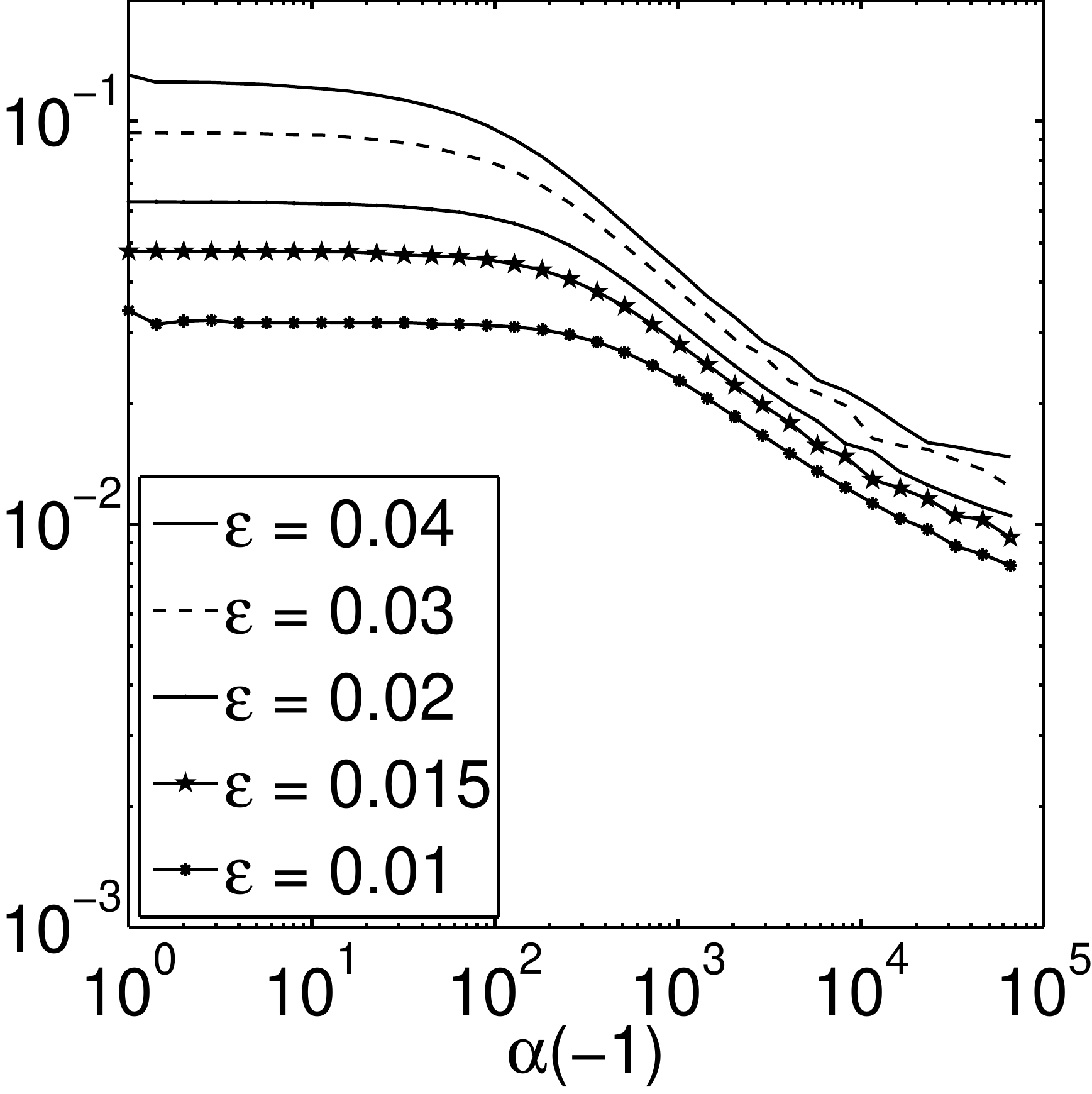,width=0.3\textwidth}\hfill
  \epsfig{file=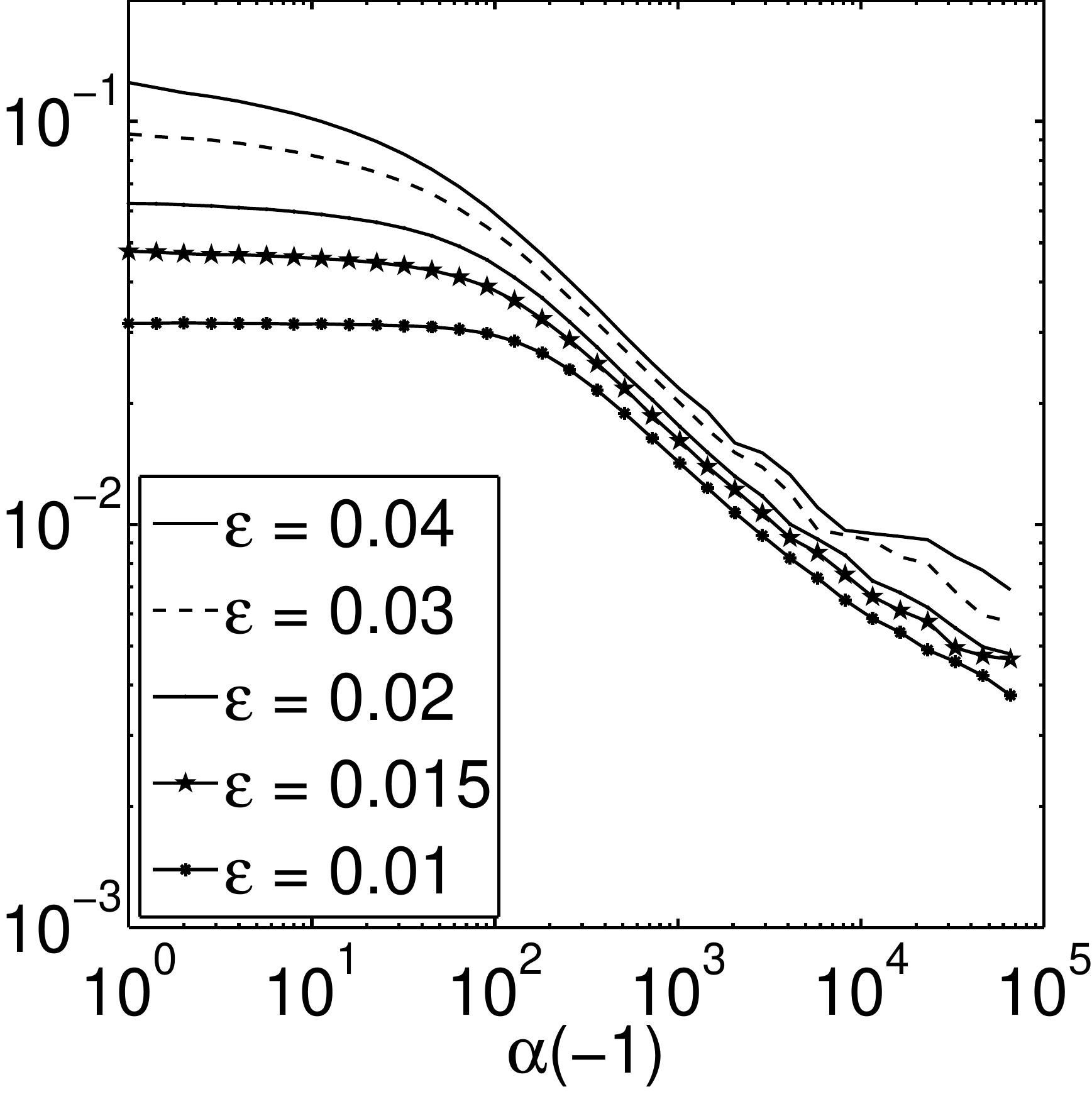,width=0.3\textwidth}
  \caption{Size of the scaled interfacial area for various combinations of 
  $\overline{\alpha} =  \alpha(-1)$ and $\epsilon$, 
  with $\gamma=0.5$ (left), $\gamma=0.05$ (middle) and $\gamma=0.005$  (right).}
  \label{fig:num:atEps}
\end{figure}

We observe that there is a regime of values for $\overline \alpha$ where the interfacial width only
depends on $\epsilon$. But we also see, that,
depending on $\gamma$ and $\epsilon$, there is a
regime where the interfacial width scales like $\overline \alpha^\kappa$ with some $\kappa\in
\mathbb{R}$ which depends on $\gamma$.

The change in the behaviour of the interfacial width occurs at 
$\alpha(-1)\approx C(\gamma)\epsilon^{-1}$,
where $C(\gamma)$ is a constant depending linearly on $\gamma$.
This is exactly the convergence rate necessary to get
analytical convergence results, compare Remark~\ref{r:ConvergenceRateTwoDim} and \cite{hecht}.

We recall that this test is run with constant $\mu=1$ and that the results might differ for
different values of $\mu$. In particular the value of $\mu$ also has an influence on the interfacial
width through the mixing energy $\frac12|\b u_h|^2 - \b u_h \cdot \b q_h$, see Section
\ref{ssec:num:MixingEnergy}.

\subsection{A treelike structure}\label{ssec:num:TH}
In this first example we investigate how our phase field approach is able to find optimal topologies
starting from a homogeneous porous material. This example is similar to an example provided in
\cite{Hansen_Dissertation}.
The setup is as follows. 
The computational domain is $\Omega = (0,1)^2$. On the boundary we have one parabolic inlet as
described in \eqref{eq:num:parabolicBoundary} and four parabolic outlets. The corresponding
parameters are given in Table \ref{tab:num:TH:boundary}.

\begin{table}
\footnotesize
\centering
\begin{tabular}{ccccc}
direction & boundary & m & l & h\\
\hline
inflow  & $\{x\equiv 0\}$ & 0.80 & 0.2 & 3\\
outflow & $\{y\equiv 0\}$ & 0.80 & 0.1 & 1\\
outflow & $\{y\equiv 1\}$ & 0.65 & 0.1 & 1\\
outflow & $\{x\equiv 1\}$ & 0.70 & 0.2 & 1\\
outflow & $\{x\equiv 1\}$ & 0.25 & 0.2 & 1
\end{tabular}
\caption{Boundary data for the treelike structure.}
\label{tab:num:TH:boundary}
\end{table}

We use $\gamma = 0.01$ and $\mu = 0.01$. 
For $\alpha_\epsilon$ we start with $\overline \alpha = 5$ and increase it later.
The phase field is initialized with a homogeneous porous material $\varphi_0 = 0$.
We start with a homogeneous mesh with mesh size $2e-5$ to obtain a first guess of the
optimal topology.
After the material demixes, i.e. $\|\nabla w^h\|_{\b L^2(\Omega)}\leq 2$, we start with adapting the mesh to the
resulting structures using the adaptation procedure described in Section \ref{ssec:adaptConcept}.
For the adaptive process we use the parameter $a_{\min}= 4e-7$, $a_{\max}=0.01$, $\theta^r = 0.1$,
and $\theta^c=0.05$.
As soon as $\|\nabla w^h\|_{\b L^2(\Omega)}\leq 1$ holds we start with increasing
$\overline \alpha$ to $\overline \alpha = 50$ and stop the allover procedure as soon as
$\overline\alpha = 50$ and $\|\nabla w^h\|_{\b L^2(\Omega)}\leq 1e-5$ holds.

In Figure \ref{fig:num:TH:evolution} we depict the temporal evolution of the optimization process.
The images are numbered from top left to bottom right.
Starting from a homogeneous distribution of porous material, we see that the inlet and the outlets
are found after very few time instances and that the main outlets on the right and the inlet 
are connected after only a few more time steps. 
At the bottom left of the computational domain we first obtain finger like structures that
thereafter vanish. 
We note that, due to the porous material approach, not all outlets are connected with the inlet
during the whole computation.
At the final stage of the optimization the evolution slows down and we end 
with the topology depicted at the bottom right after 188 time steps of simulation.

\begin{figure}
  \centering
  \hfill
  \epsfig{file=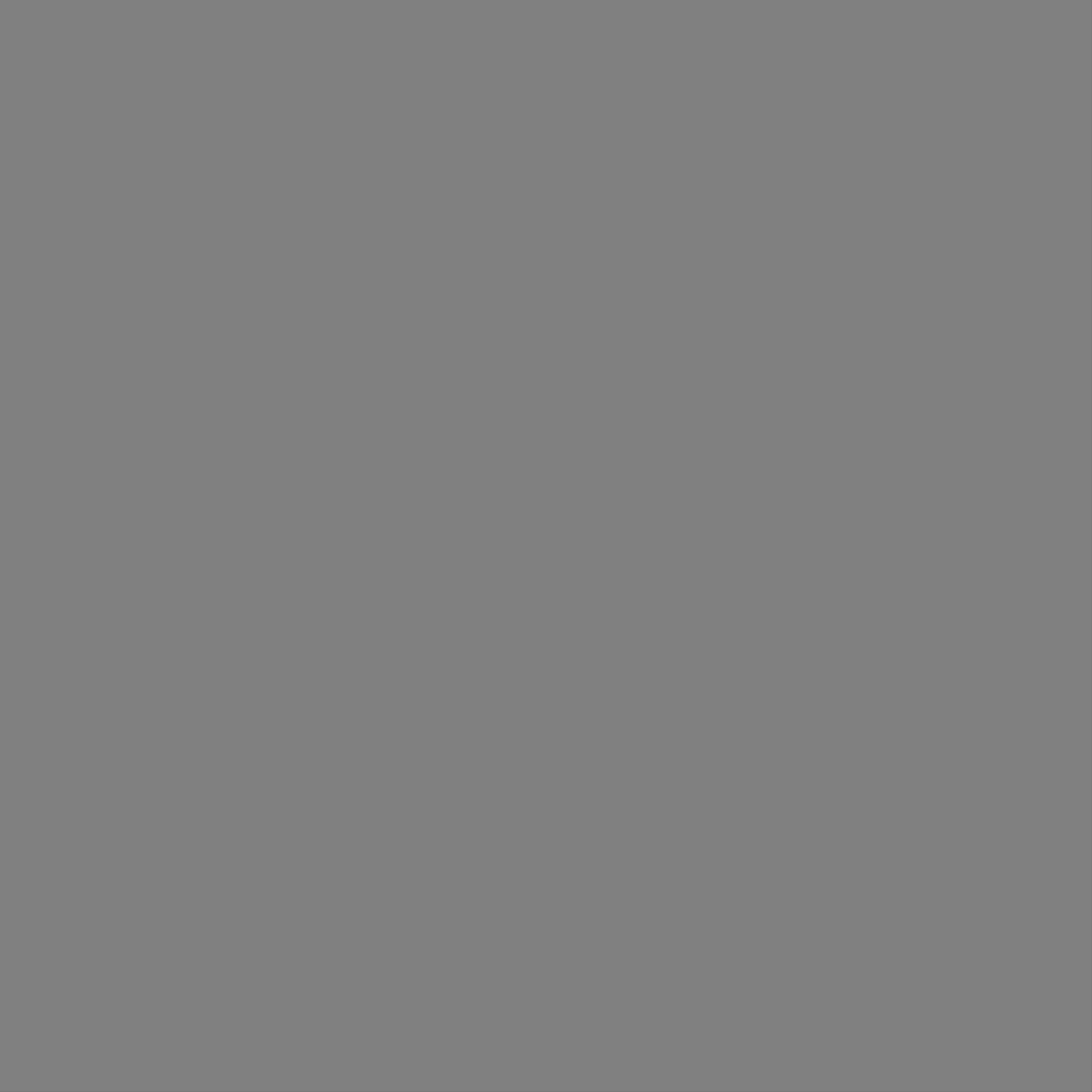,width=0.29\textwidth}
  \hfill
  \epsfig{file=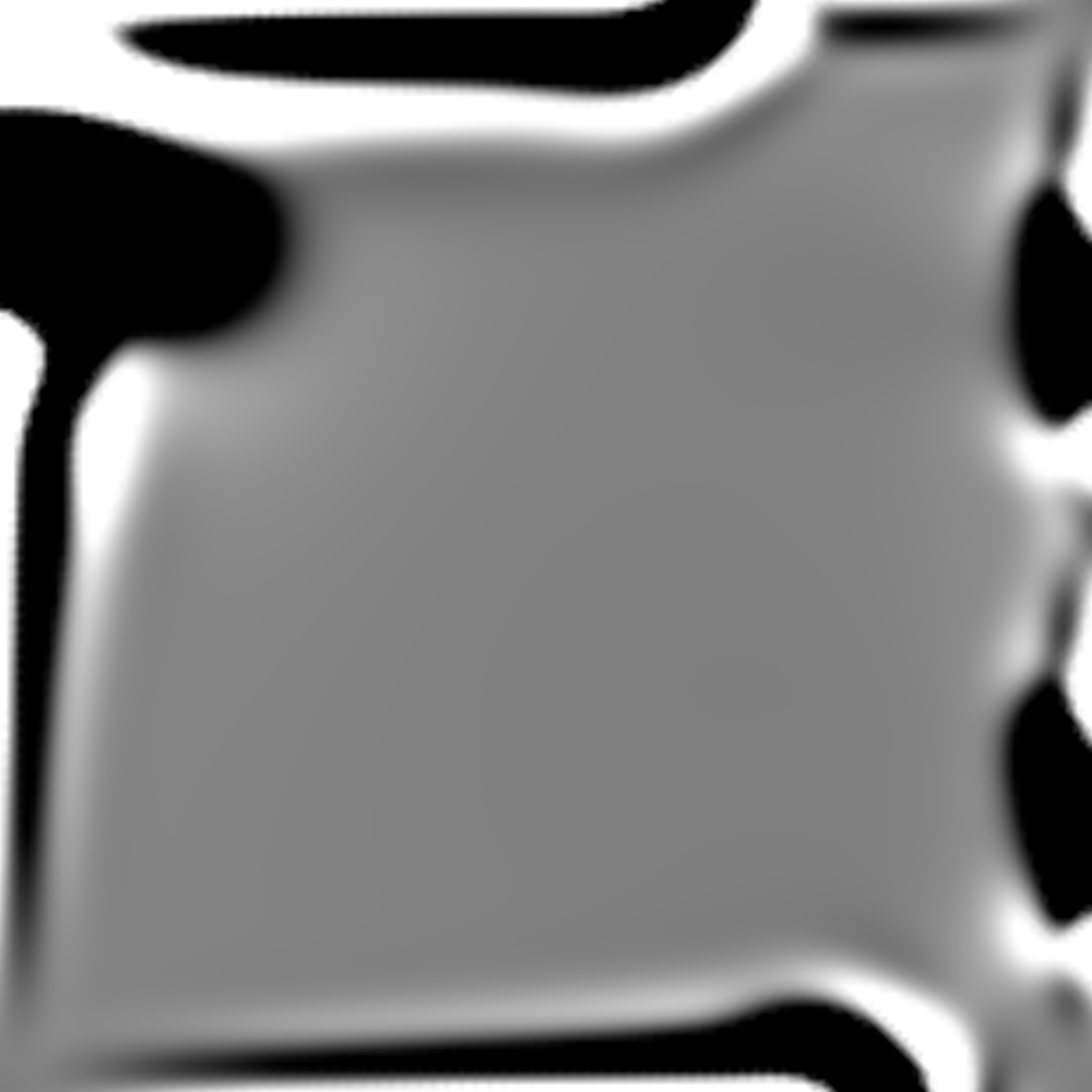,width=0.29\textwidth}
  \hfill
  \epsfig{file=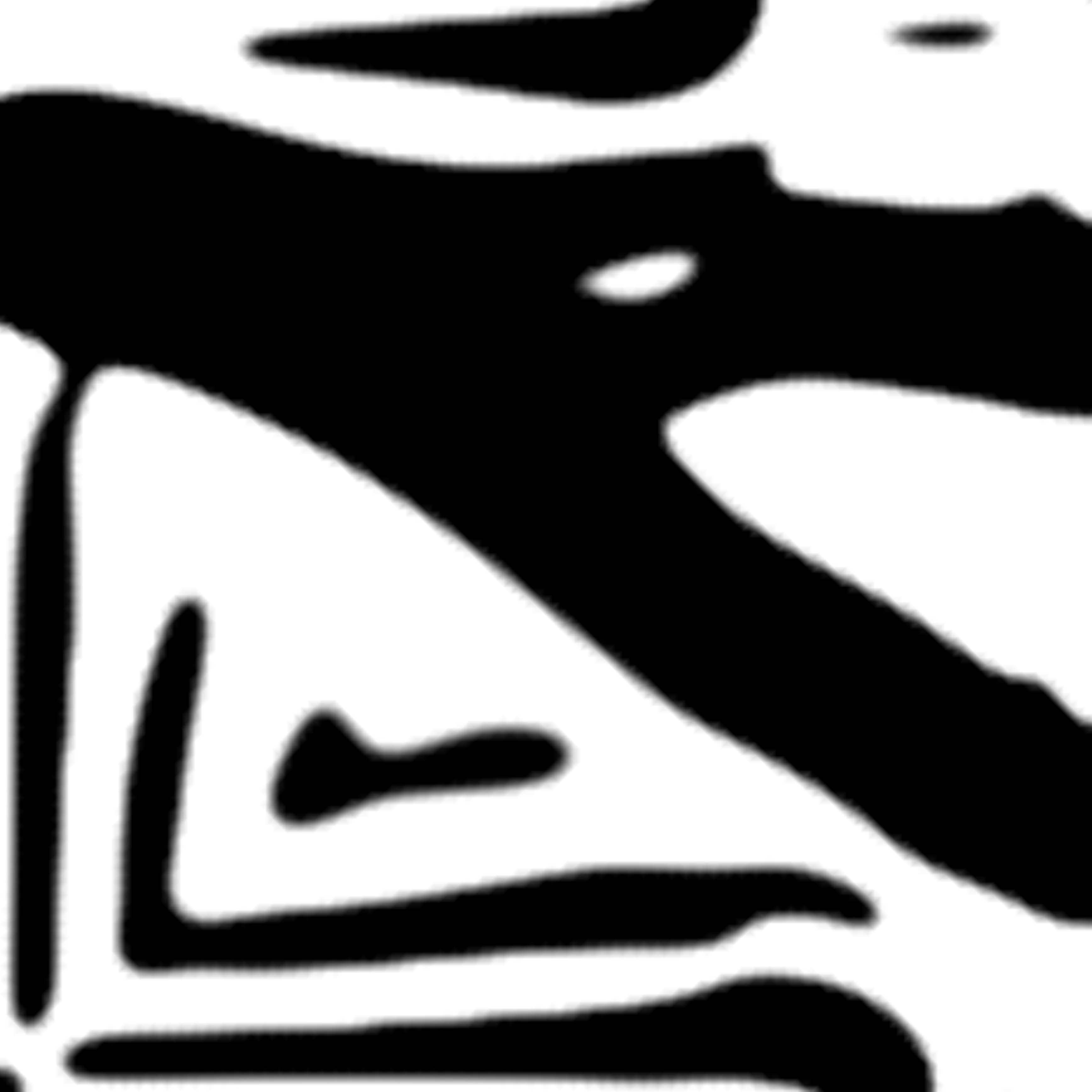,width=0.29\textwidth}
  \hfill\\[2ex]
  \hfill
  \epsfig{file=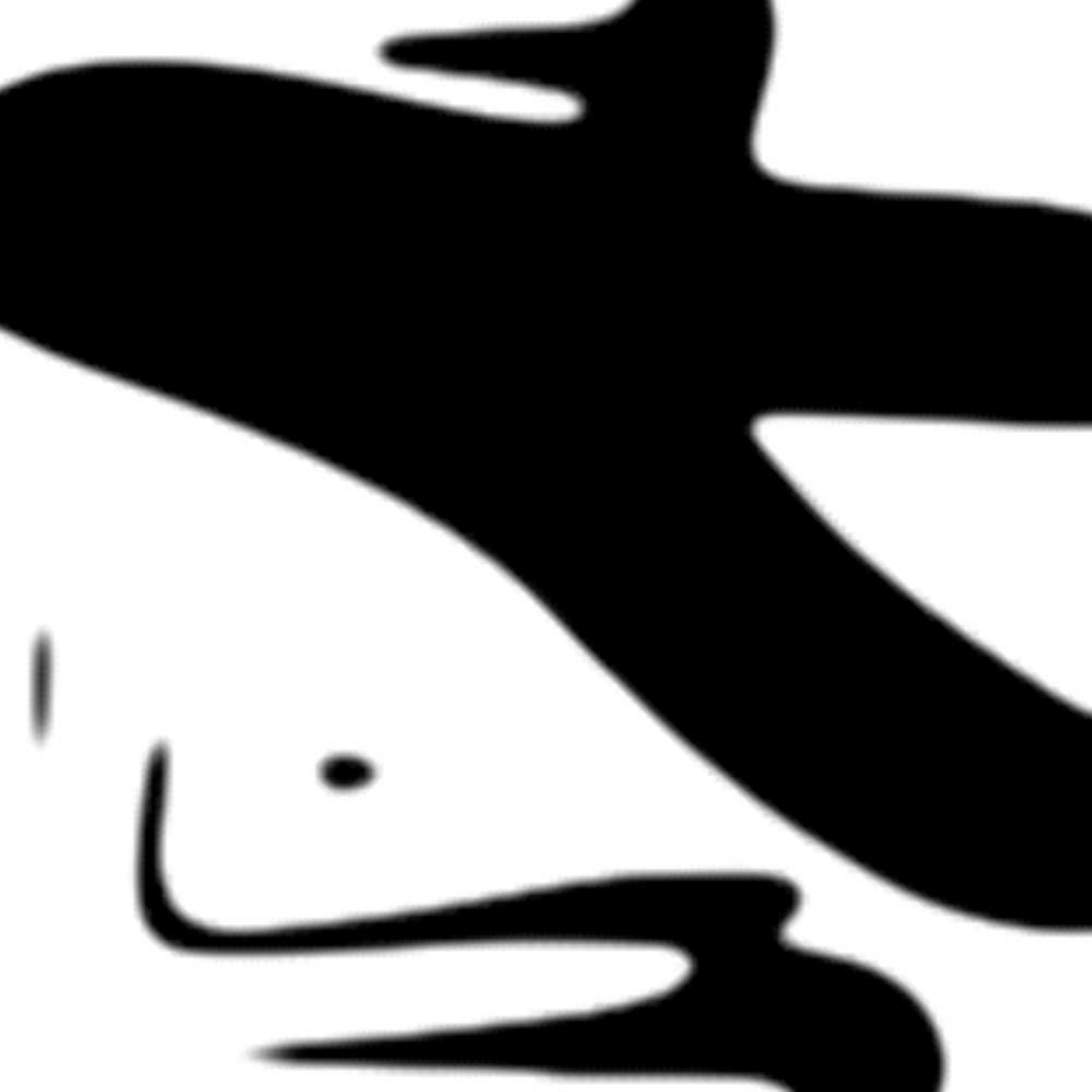,width=0.29\textwidth}
  \hfill
  \epsfig{file=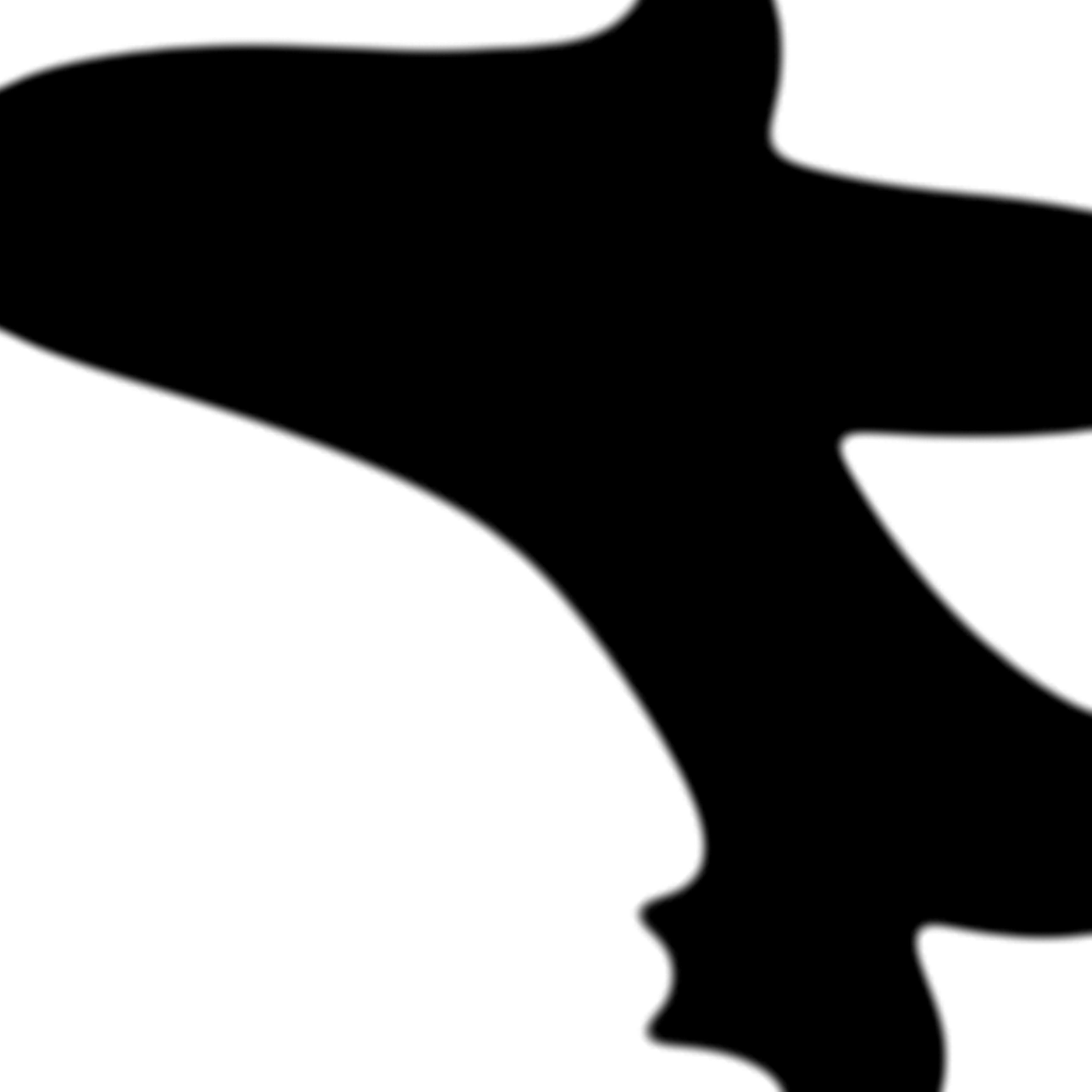,width=0.29\textwidth}
  \hfill
  \epsfig{file=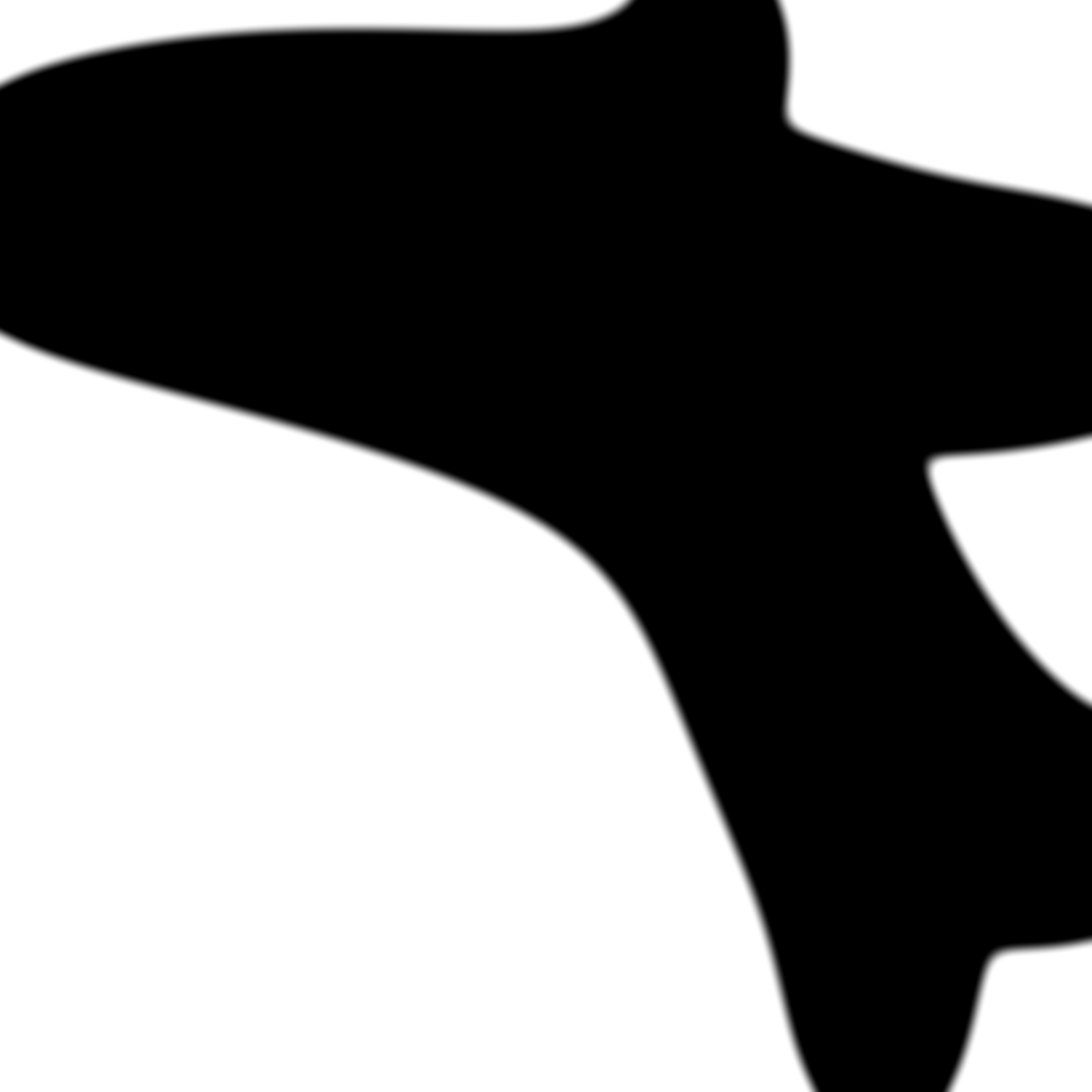,width=0.29\textwidth}
  \hfill
  \caption{The initial phase field $\varphi_0$ for the treelike structure and the phase field after
  6, 12, 36, 70, and 188 time steps (top left to bottom right).}
  \label{fig:num:TH:evolution}
\end{figure}

\subsection{A rugby ball}\label{ssec:num:RB}
We next investigate the overall behaviour of the adaptive concept and give an example showing the
influence of the parameters $\gamma$ and $\mu$ on the interfacial area. 
The aim is to optimize the shape of a ball in an outer flow as is investigated in
\cite{borrvall,pingen_TopoOpt_Boltzmann,Schmidt_shape_derivative_NavierStokes}.

In the computational domain $\Omega = (0,1)\times(0,5)$ we have a circle located at $M = (0.5,0.5)$ with
radius $r = \sqrt{(10\pi)^{-1}}$. On the boundary $\partial \Omega$ we impose Dirichlet data $\b
g\equiv(0,1)^T$ for the Navier--Stokes equations.
The domain is chosen large enough to neglect the
influence of the outflow boundary on the optimized topology.

In \cite{borrvall} it is shown that for Stokes flow the optimal topology equals a rugby ball, while
in \cite{pingen_TopoOpt_Boltzmann,Schmidt_shape_derivative_NavierStokes} the authors obtain an
airfoil-like shape for Navier--Stokes flow and small values of $\mu$.
The parameters used here are $\epsilon=0.005$ and $\overline{\alpha}=50$. For the adaptive concept
we fix $\theta^r = 0.2$, $\theta^c = 0.05$, $a_{\min}=10^{-7}$ and $a_{\max} = 5 \cdot 10^{-4}$. As
initial mesh we use a homogeneous mesh with mesh size $a_{\text{init}} = 1/1600$ and refine
the region $|\varphi_0|\leq 1$ to the finest level, where $\varphi_0$ denotes the initial phase field.

\subsubsection{Optimal shapes for various $\gamma$ and $\mu$}
We start with depicting our numerical findings for various values of $\gamma$ and $\mu$.
Here we proceed as follows.
We optimize the shape for decreasing values of $\gamma \in [10^{-4},10]$ and $\mu = 1$.
The optimal geometry for $\mu=1$ and $\gamma=10^{-4}$ thereafter is used as initial value for
decreasing $\mu \in [500^{-1},1]$ while $\gamma=10^{-4}$ is kept fix.
In Figure \ref{fig:num:RB:Results_gamma} we depict the optimal shapes for $\mu=1$ and
$\gamma \in \{10,0.1,0.01,0.0001\}$, and
in Figure \ref{fig:num:RB:Results_RE} we depict the optimal shapes for $\gamma = 10^{-4}$
and $\mu \in \{10^{-1},100^{-1},300^{-1},500^{-1}\}$.

\begin{figure}
  \centering
  \fbox{
  \epsfig{file=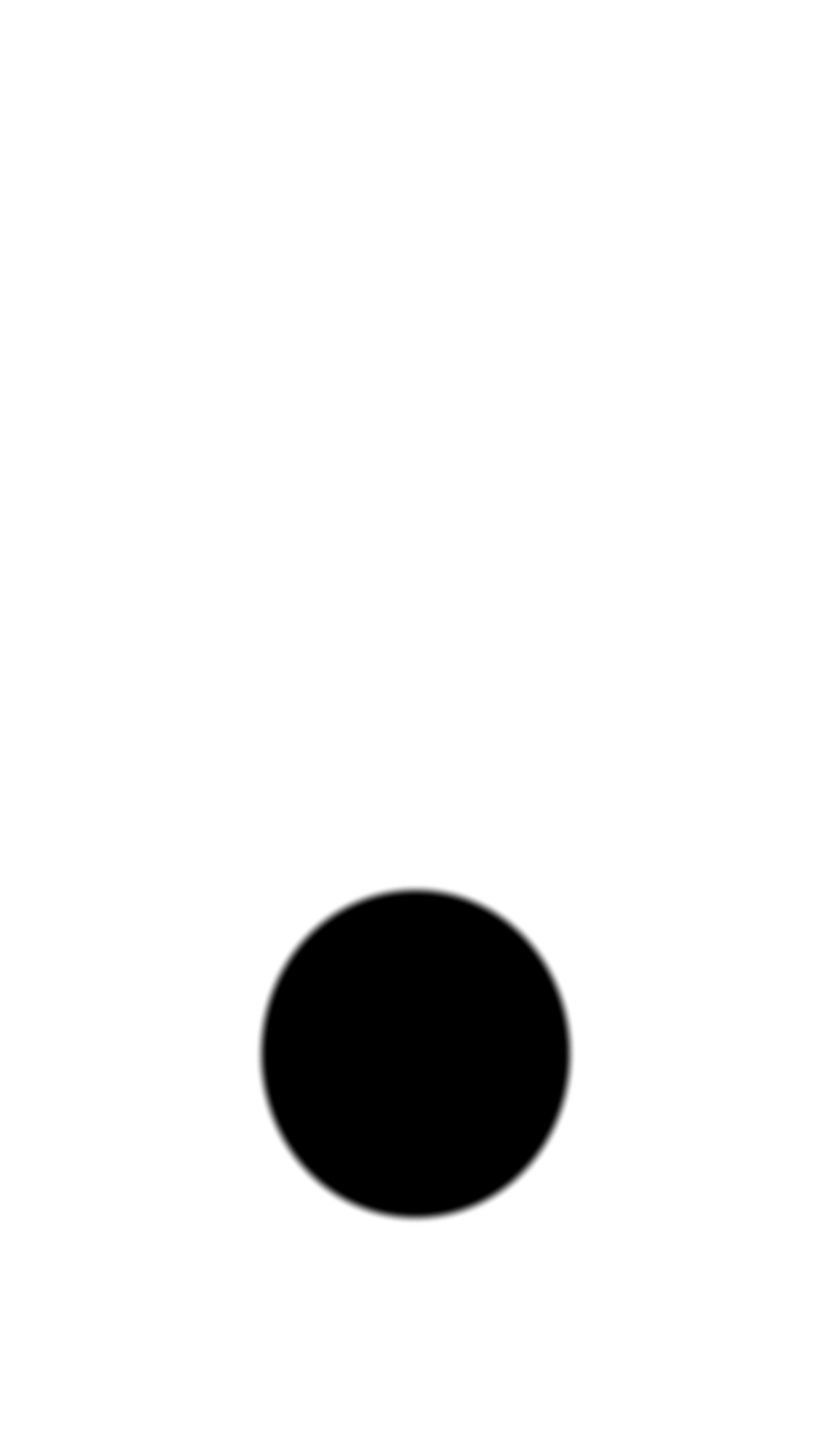, width=0.20\textwidth}
  }
  \hfill
  \fbox{
  \epsfig{file=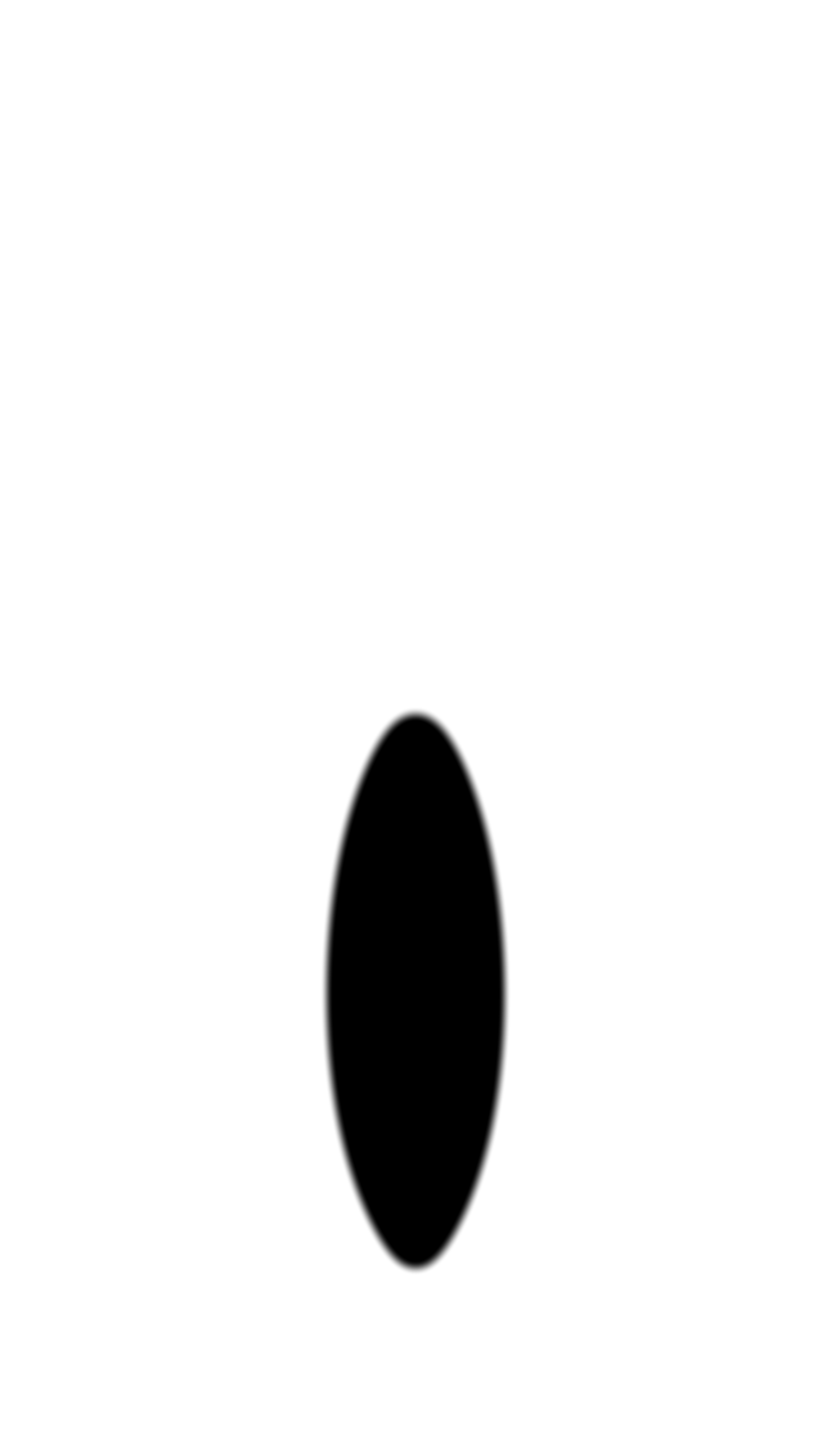, width=0.20\textwidth}
  }
  \hfill
  \fbox{
  \epsfig{file=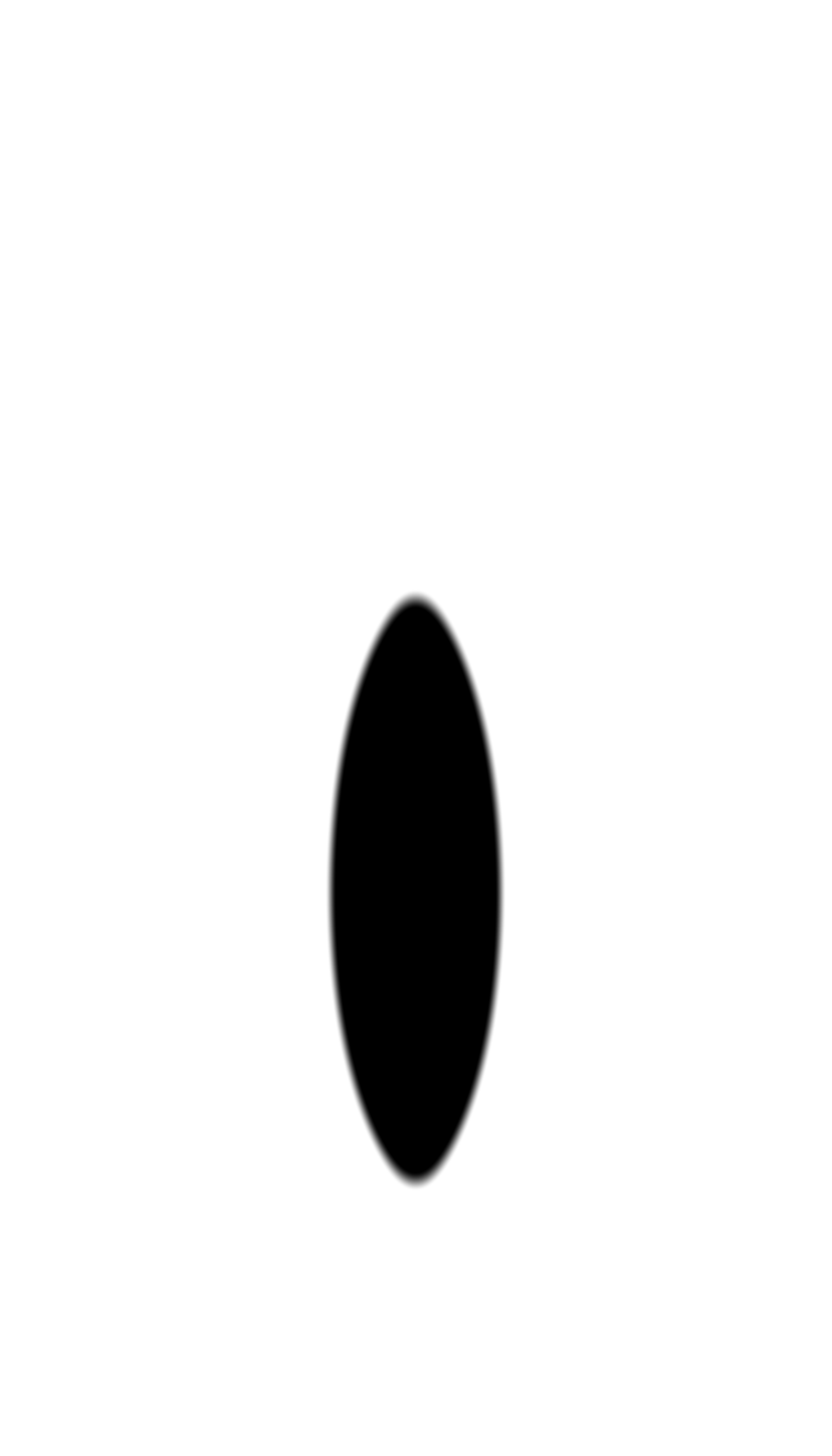, width=0.20\textwidth}
  }
  \hfill
  \fbox{
  \epsfig{file=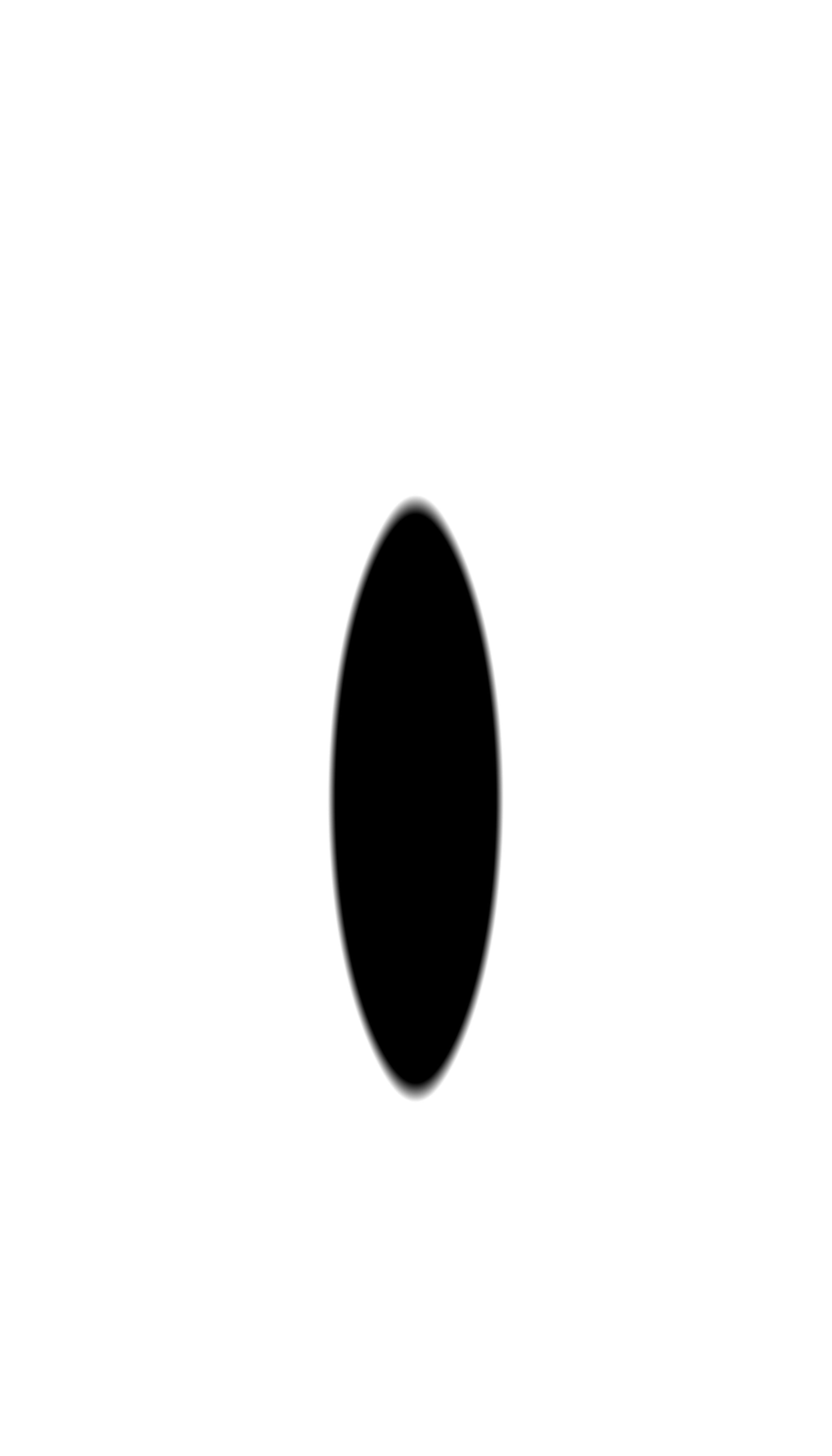, width=0.20\textwidth}
  }
  
  \caption{Optimal topologies for the rugby ball example for $\mu=1$ and
  $\gamma \in \{10, 0.1, 0.01, 0.0001\}$ (left to right).}
  \label{fig:num:RB:Results_gamma}
\end{figure}

\begin{figure}
  \centering
  \fbox{
  \epsfig{file=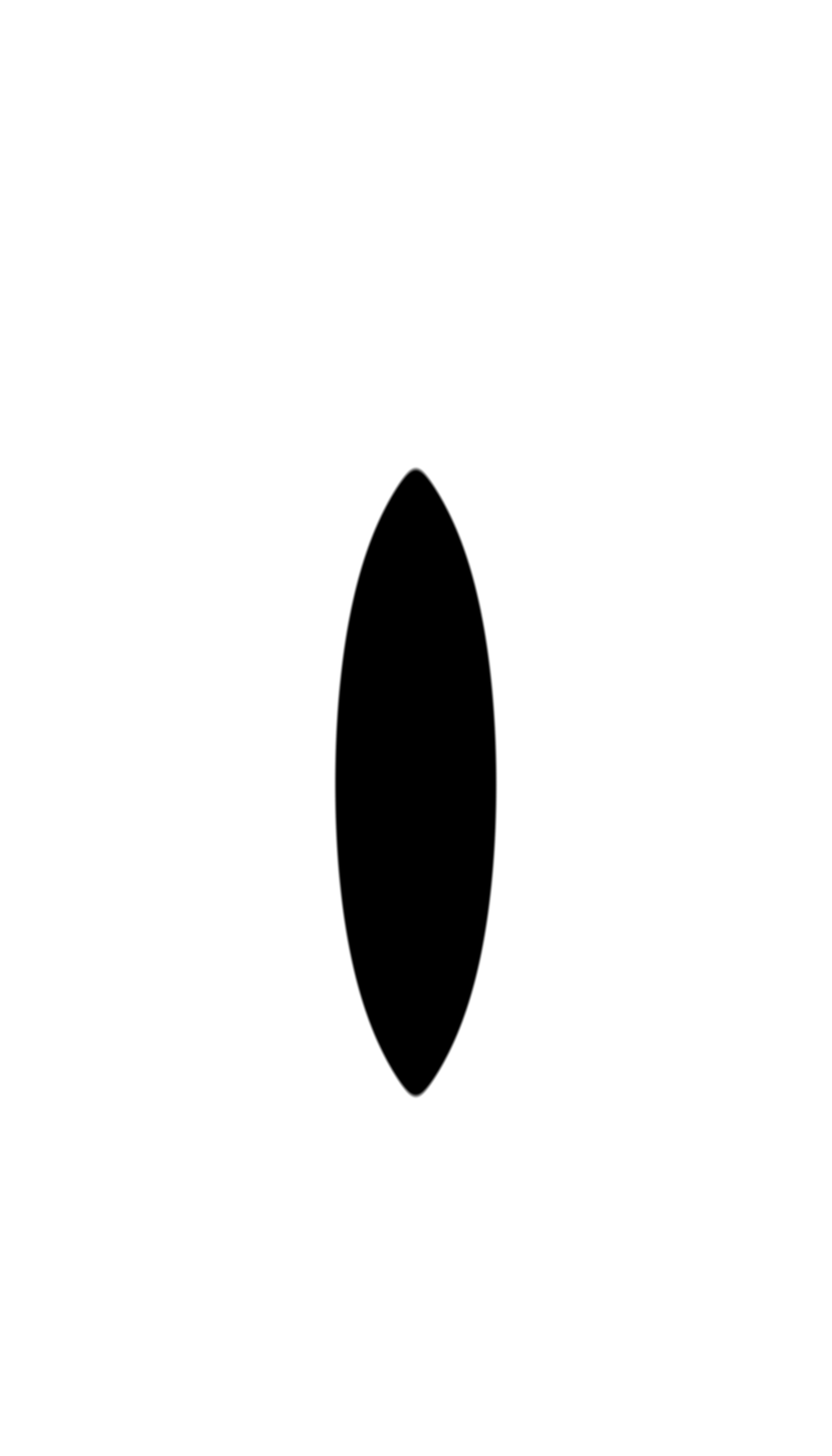, width=0.20\textwidth}
  }
  \hfill
  \fbox{
  \epsfig{file=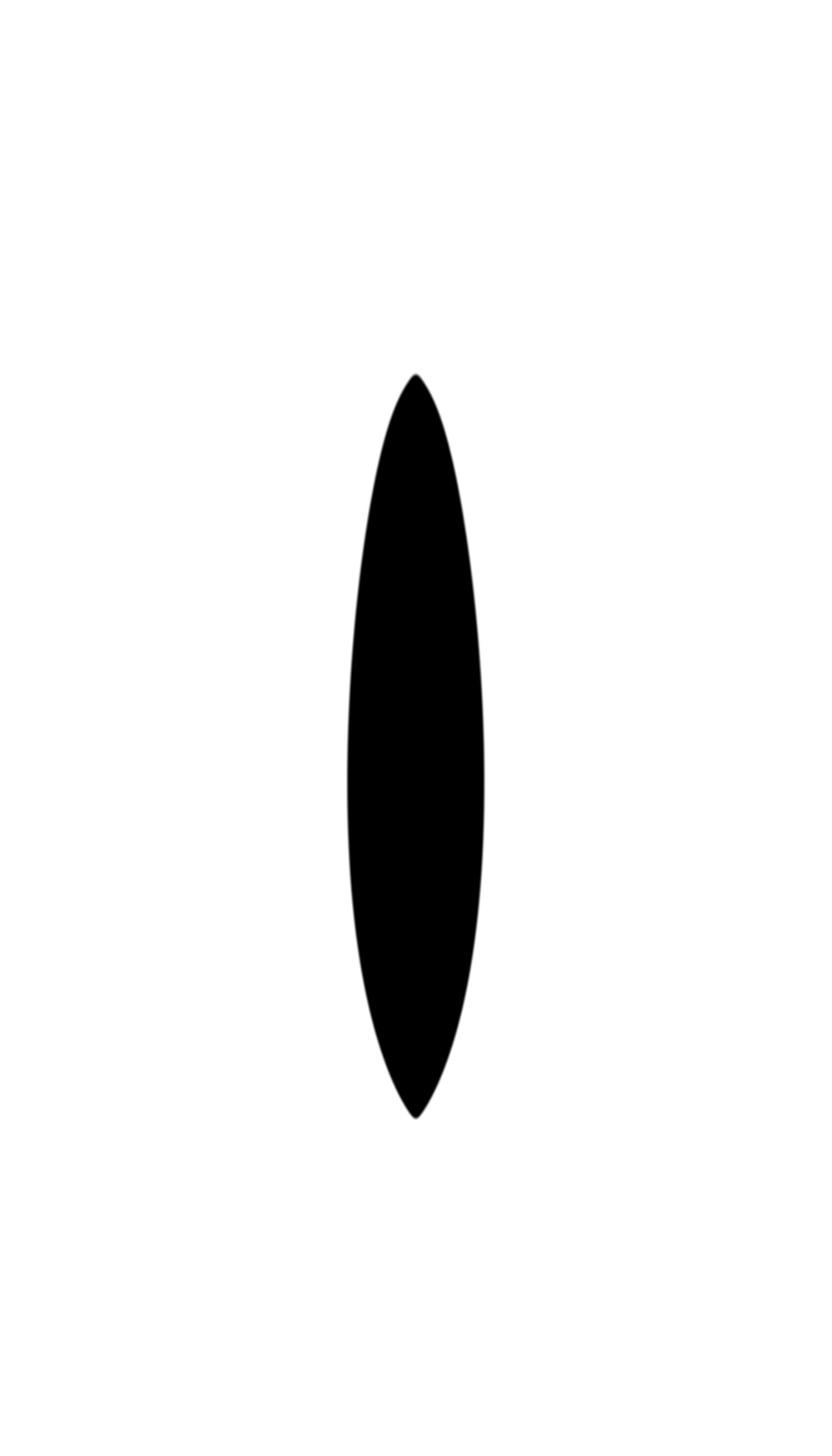, width=0.20\textwidth}
  }
  \hfill
  \fbox{
  \epsfig{file=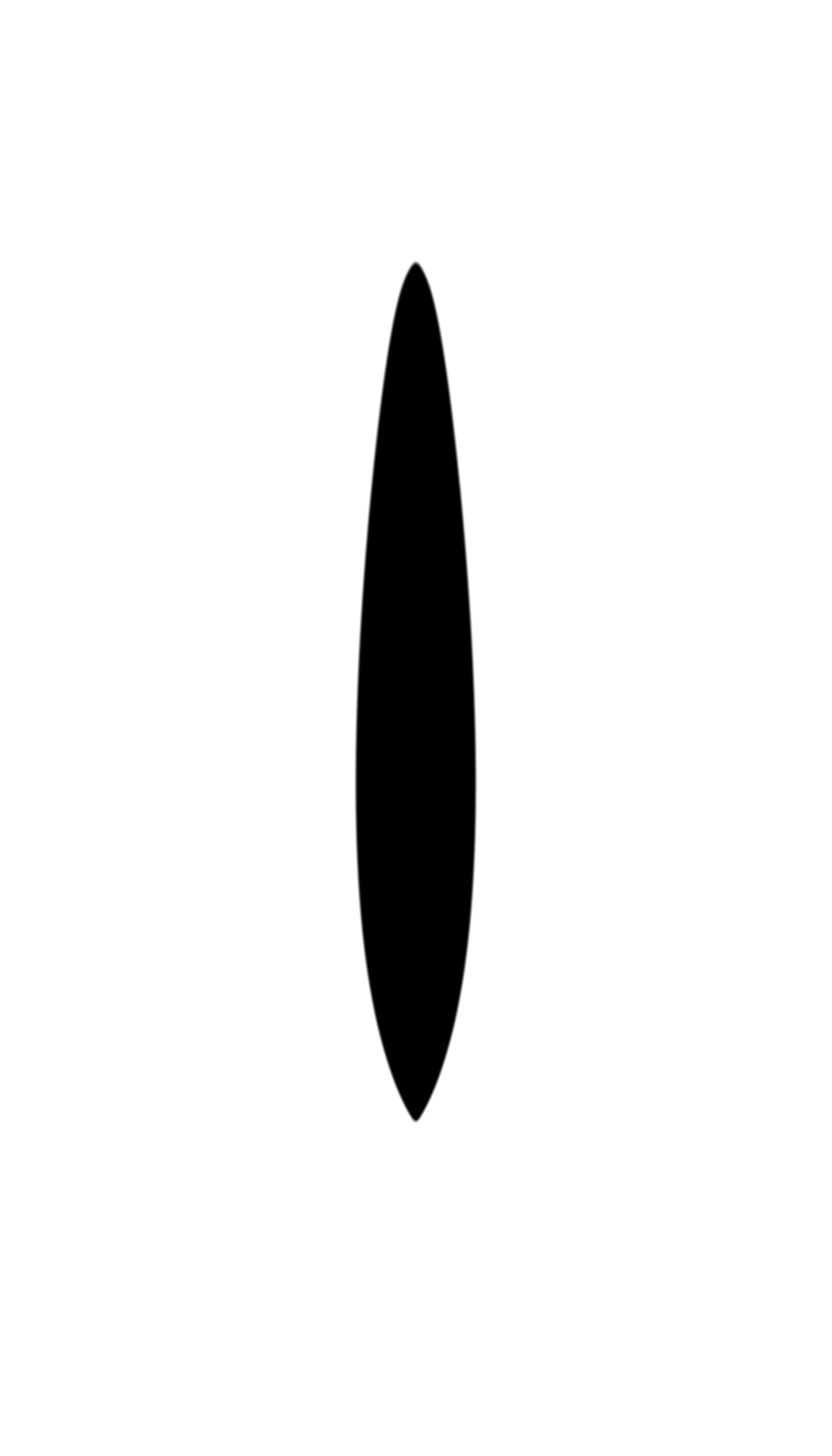, width=0.20\textwidth}
  }
  \hfill
  \fbox{
  \epsfig{file=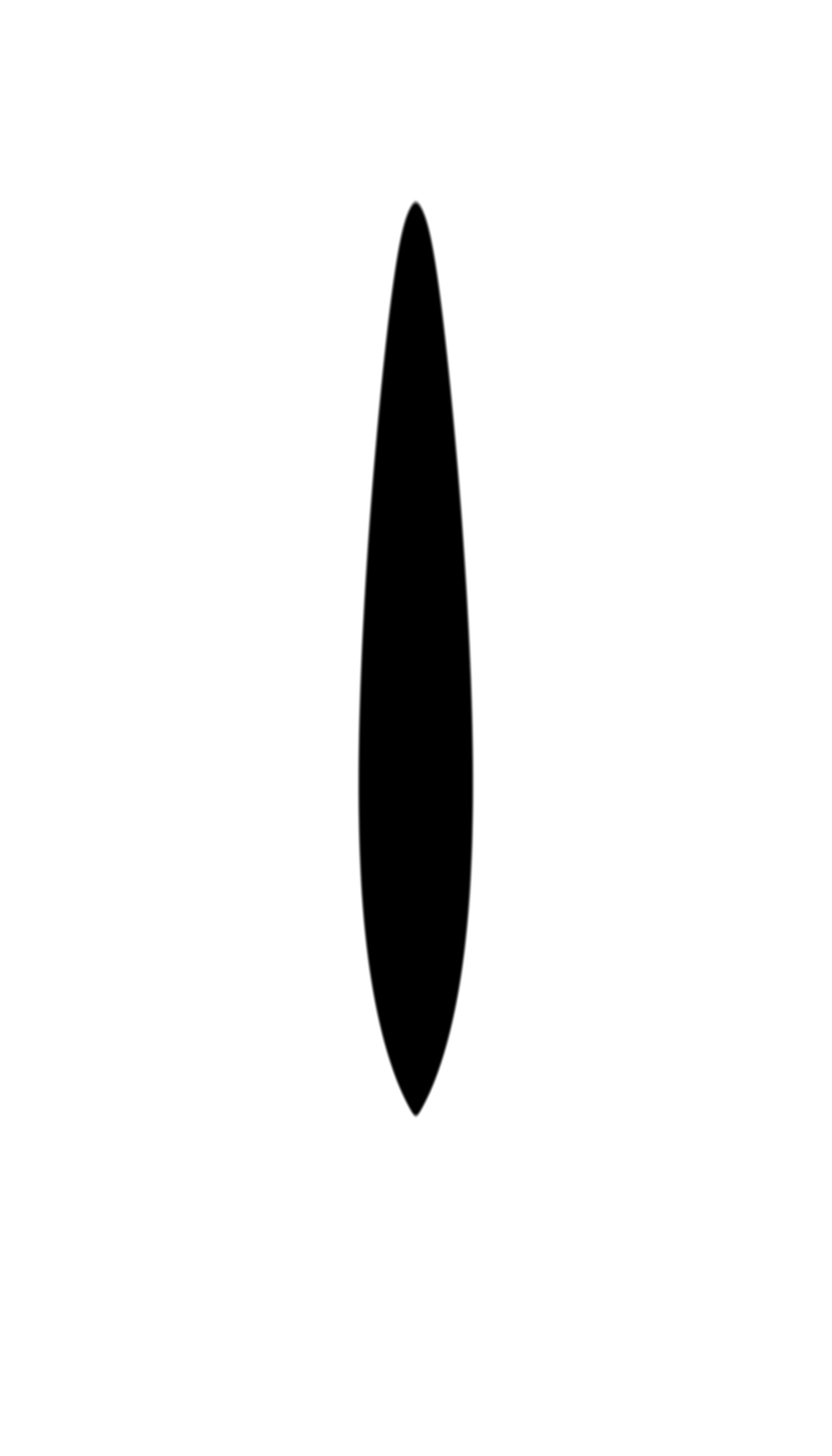, width=0.20\textwidth}
  }
  \caption{Optimal topologies for the rugby ball example for $\gamma=10^{-4}$ and
  $\mu \in \{10^{-1}, 100^{-1}, 300^{-1}, 500^{-1}\}$ (left to right).}
  \label{fig:num:RB:Results_RE}
\end{figure}

We see that for large values of $\gamma$ the Ginzburg--Landau energy dominates the minimizing
problem and thus we obtain an optimal shape which is close to a circle. With $\gamma$ getting smaller we obtain shapes that resemble
rugby balls like shapes as obtained in \cite{borrvall} for the Stokes flow.
In particular we see that the top and bottom tip get sharper as we decrease the value of $\gamma$.
This can be explained by the Ginzburg--Landau
energy. This term penalises the interfacial size and explains why for large values
of $\gamma$ the optimal shape is close to a circle. 
Note that the optimal shape can locate freely in
the computational domain and therefore the optimal shape for $\gamma=10^{-4}$ has a slightly larger
distance to the bottom boundary than the optimal shapes for larger $\gamma$.

As argued in \cite{pingen_TopoOpt_Boltzmann}, for $\mu$ taking smaller values, the optimal shape
tends to an airfoil. This is what we observe in our numerics, 
see  Figure \ref{fig:num:RB:Results_RE}.

For a quantitative description of the optimal shapes we follow 
\cite[Rem. 12]{Schmidt_shape_derivative_NavierStokes} and introduce the friction drag of an obstacle
in free flow as
\begin{equation}\label{eq:num:drag}
  F_D = 
  \int_{\{\varphi=0\}} -\mu \left(\left(\nu\cdot\nabla\right)\b u\right)\cdot a + p \nu\cdot  a\ds.
\end{equation}
Here $\nu$ is the unit normal on the boundary of the ball pointing inwards and $a$ is the direction of attack of
the flow field.
In our example we have $a = (0,1)^T$ since the flow is attaining from the bottom.
By using the Gauss theorem we write $F_D$ as an integral over the ball given by $\varphi <0$ and
obtain
\begin{equation}\label{eq:num:drag_gauss}
  F_D =-\int_{\{\varphi<0\}} \div \left( -\mu \nabla u_2 + (0,p)^T \right)\dx
  =  -\int_{\{\varphi<0\}}-\mu \Delta u_2 + p_y\dx.
\end{equation}
Note that the normal $\nu$ in \eqref{eq:num:drag} points into the rugby ball and thus we obtain the
minus sign in \eqref{eq:num:drag_gauss}. Here $u_2$ denotes the second component of the velocity
field $\b u$ and $p_y$ denotes the derivative of $p$ in $y$-direction.

As second comparison value we define the circularity of the rugby ball. This value is introduced in
\cite{Wadell_Circularity} to describe the deviation of circular objects from a circle.
It is defined by
\begin{equation}\label{eq:num:circularity}
  \theta =
  \frac{\mbox{Circumference of circle with same area}}
  {\mbox{Circumference of object}}
  =
  \frac{\sqrt{4\pi\int_{\{\varphi<0\}}\dx}}
  {\int_{\{\varphi=0\}}\ds}
  \leq 1,
\end{equation}
where a value of $\theta \equiv 1$ indicates a circle.

In Table \ref{tab:num:RB:dissPow_drag_circ} we give results for our numerical findings.
\begin{table}
\centering
\footnotesize 
\begin{tabular}[t]{ccccc}
$\gamma$ & $\mu$ & $F$ & $\theta$ & $F_D$\\
\hline 
10.0000&   1&   7.2266&   0.9996&  21.6140\\
 1.0000&   1&   6.5317&   0.9664&  18.5820\\
 0.1000&   1&   6.1828&   0.8005&  16.6710\\
 0.0100&   1&   6.1494&   0.7722&  16.4640\\
 0.0010&   1&   6.1480&   0.7681&  16.4510\\
 0.0001&   1&   6.1427&   0.7674&  16.4310
\end{tabular}
\begin{tabular}[t]{ccccc}
$\gamma$ & $\mu$ & $F$ & $\theta$ & $F_D$  \\
\hline
0.0001&  10$^{-1}$&   1.1353&   0.7335&   2.3596\\
0.0001& 100$^{-1}$&   0.1830&   0.6349&   0.3244\\
0.0001& 200$^{-1}$&   0.1188&   0.5901&   0.1910\\
0.0001& 300$^{-1}$&   0.0942&   0.5568&   0.1395\\
0.0001& 400$^{-1}$&   0.0805&   0.5403&   0.1114\\
0.0001& 500$^{-1}$&   0.0715&   0.5253&   0.0930
\end{tabular}
\caption{Comparison values for the rugby example. $F$ is the dissipative power, $\theta$ denotes
the circularity, and $F_D$ the drag force. The optimization aim is the minimization of the
dissipative power.}
\label{tab:num:RB:dissPow_drag_circ}
\end{table}
As discussed above for large values of $\gamma$ the Ginzburg--Landau energy dominates the functional
under investigation. This results in optimal shapes that are close to circles as can be seen for
$\gamma=1$ and $\gamma=1$ where we have $\theta=1$ and $\theta=0.97$ respectively. We further see
that for $\gamma=0.01$ the optimal shape is determined by the dissipative power, since
the results for $\gamma=0.01$ and $\gamma=0.0001$ are very close together.
Concerning the dependence with respect to $\mu$ we see how the dissipative energy, which scales with
$\mu$, decreases with decreasing $\mu$. We also obtain that both the circularity and the drag are
reduced for smaller values of $\mu$. For the drag we have approximately $F_D \sim \mu^{0.84}$.

\subsubsection{Behaviour of the adaptive concept}
Next we investigate the behaviour of the adaptive concept. Since the error indicators  only contain
the jumping terms of the gradient, 
we expect the indicators mainly to be located at the borders of the interface, i.e. the isolines
$\varphi=\pm 1$.
In Figure \ref{fig:num:RB:ErrorEst} we depict the distribution of the error indicator
$\eta_{T_E}$ for the optimal topology for $\gamma=10^{-4}$ and $\mu=1$.

\begin{figure}
  \centering
  \fbox{
  \epsfig{file=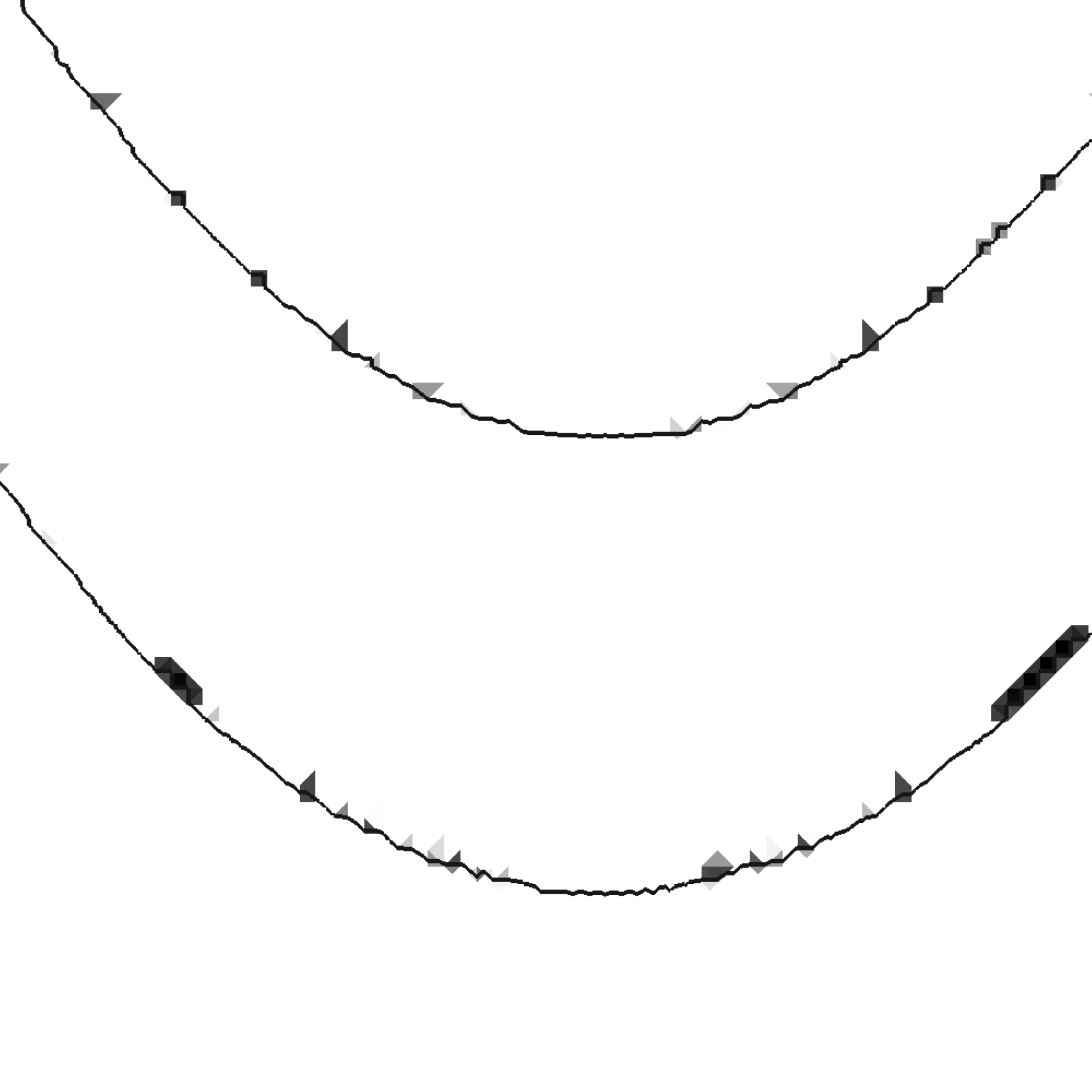, width=0.40\textwidth}
  }
  \fbox{
  \epsfig{file=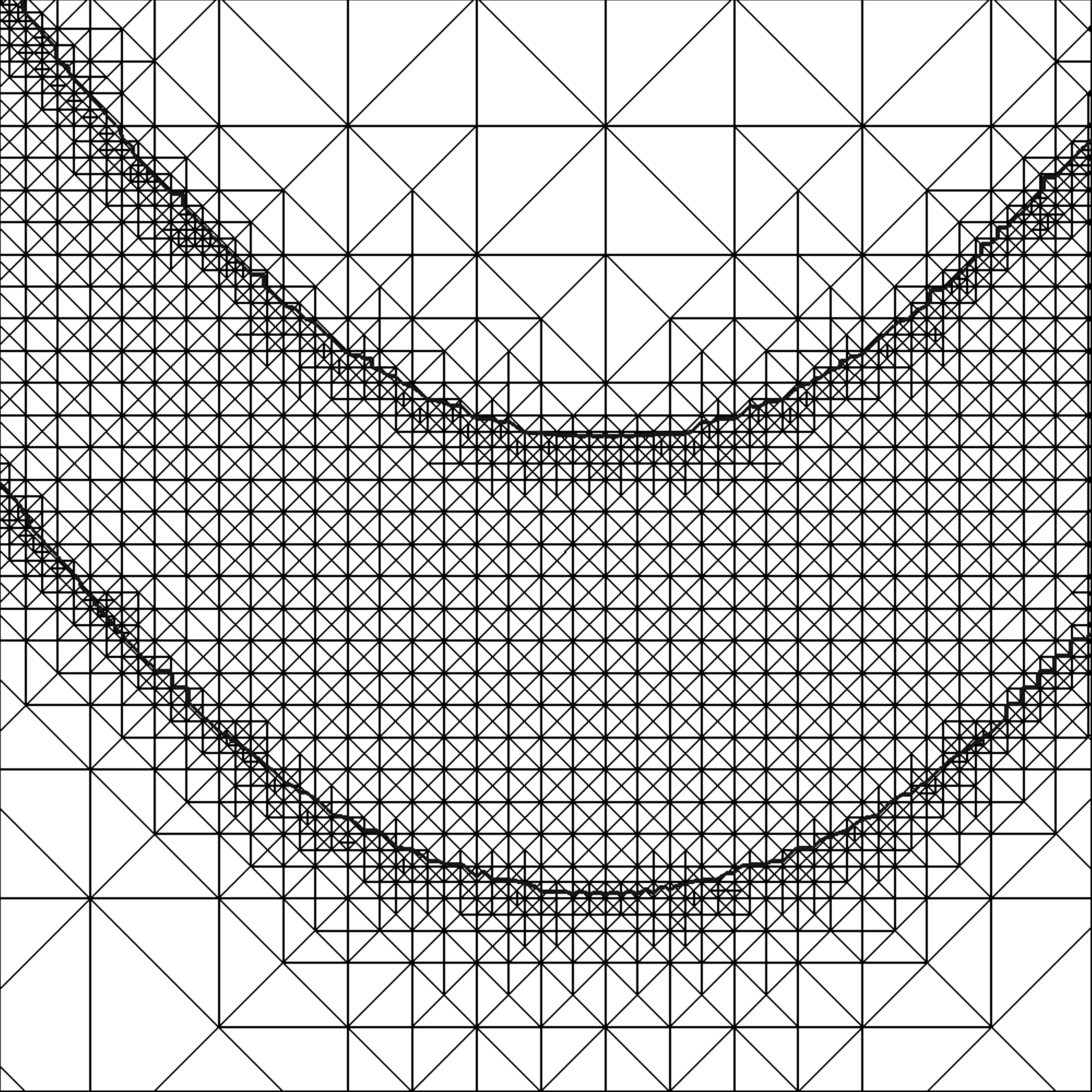, width=0.40\textwidth}
  }
  \caption{The bottom arc of the rugby ball for $\gamma=10^{-4}$ and $\mu=1$.
  The distribution of
  $\eta_{T_E}$ across the interface is shown in the left plot, where
  darker areas indicate larger error.
  The spatial resolution of the interface is depicted in the right plot.
  The bold lines   indicate the discrete sets $\varphi\equiv \pm1$. }
  \label{fig:num:RB:ErrorEst}
\end{figure}

We observe from the left plot, that the indicator $\eta_{T_E}$ is concentrated at the discrete
isolines $\varphi=\pm 1$.
Here the mesh is refined to the finest level as we see in the right plot. Inside the interface the
triangles are only mildly refined. Here the phase field tends to be linear and thus a high spatial
resolution is not required to get a well resolved phase field.

\subsubsection{A view on mixing energy} \label{ssec:num:MixingEnergy}
From the point of view of Cahn--Hilliard theory, \eqref{eq:MY:CahnHilliard} alone for fixed
vector fields $\b u,\b q$ can be regarded as the Cahn--Hilliard system with a free energy $F$ given
by
\begin{align}
  F(\varphi) = \frac{\gamma}{\epsilon}(1-\varphi^2)
  + \frac{\hp}{2}\lambda^2(\varphi)
  + \alpha_\epsilon(\varphi)\left(\frac{1}{2}|\b u|^2 - \b u \cdot \b q\right).
  \label{eq:num:RB:mixingEnergy}
\end{align}
The term $\frac{1}{2}|\b u|^2 - \b u \cdot \b q$ is assumed to be
non negative.
For $F$ we require two distinct minima located at $\approx \pm 1$. If $|\varphi|>1$ holds it is
reasonable to assume that $\hp \lambda^2(\varphi)$ is the dominating term and in the ongoing we
investigate the distribution of $F$ inside the interface defined by $|\varphi|\leq 1$.
We note that this distribution in fact depends on $\gamma,\epsilon,\overline \alpha$, and $\mu$.
As in Section \ref{ssec:num:epsilon} we fix $\epsilon$ to be $0.005$. Since both $\overline \alpha$ and
$\gamma$ give a weighting of the two energy terms we also fix $\overline \alpha \equiv 50$ as
proposed in Section \ref{ssec:num:alphaepsilon}.
Thus the free parameters in this investigation are $\gamma$ and $\mu$.

In Figure \ref{fig:num:RB:MixingEnergy}
we show the distribution of the terms
$\alpha_\epsilon(\varphi)(\frac{1}{2}|\b u|^2 - \b u\cdot \b q)$
and
$\frac{\gamma}{\epsilon}(1-\varphi^2)$
at the bottom arc of the optimized rugby ball.

\begin{figure}
  \centering
  \epsfig{file=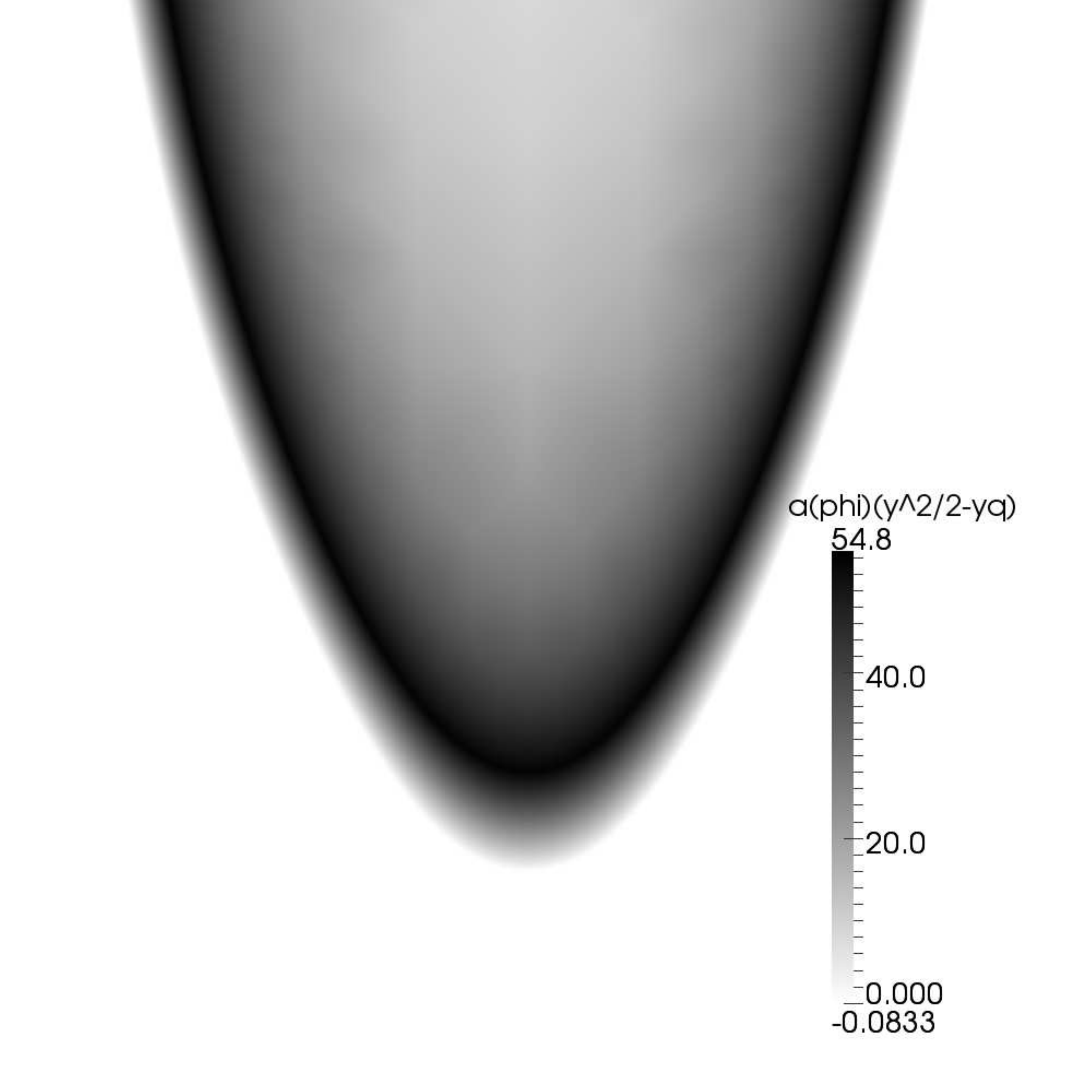,width=0.4\textwidth}
  \epsfig{file=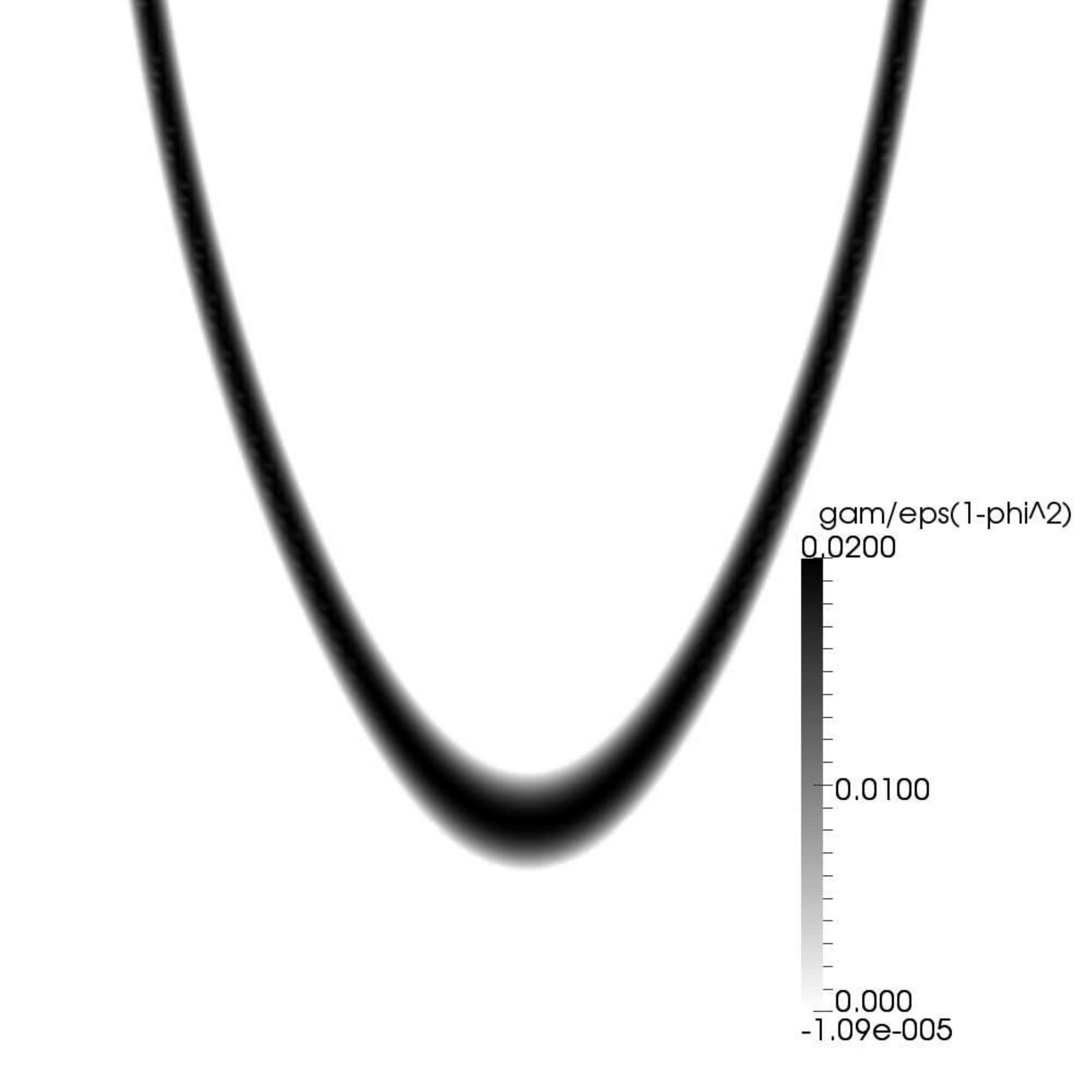,width=0.4\textwidth}
  \caption{The energies $\alpha_\epsilon(\varphi)(\frac{|\b u|^2}{2} - \b u \cdot \b q)$
  (left plot) and $\frac{\gamma}{\epsilon}(1-\varphi^2)$ (right plot) at the bottom of the
  optimized topology for $\gamma=10^{-4}$ and $\mu=1$.}
  \label{fig:num:RB:MixingEnergy}
\end{figure}

We see, that the term $\alpha_\epsilon(\varphi)(\frac{1}{2}|\b u|^2 - \b u\cdot \b q)$ is
larger then $\frac{\gamma}{\epsilon}(1-\varphi^2)$ and thus dominates the demixing.
The term admits a maximum inside the interface and takes smaller values outside of the
interface. We note that the term $\frac{\gamma}{\epsilon}(1-\varphi^2)$ is symmetric across the
interface, while  $\alpha_\epsilon(\varphi)(\frac{1}{2}|\b u|^2 - \b u\cdot \b q)$ takes its maximum
near $\varphi=-1$ and especially also takes large values inside the porous material.

\subsection{An optimal embouchure for a bassoon}\label{ssec:num:FG}
As outlook we investigate the optimal shape of an embouchure for a bassoon. 
In the group of Professor Grundmann at the Technische Universit\"at  Dresden by experiments an
optimized shape was found  that has a smaller pressure loss along the pipe, while it only slightly
changes the sound of the bassoon, see \cite{grundmann}.

We apply our optimization algorithm to the problem of finding an optimal embouchure in order 
to illustrate possible fields of application of our approach. We note that again we
minimize the dissipative energy, and that we do not take further optimization constraints into
account. 

We proceed as described in Section \ref{ssec:num:TH} to find optimal shapes in
$\Omega=(0,1)^2$ for the parameters $\gamma=1e-4$, $\epsilon=0.005$, $\overline
\alpha =50$ and $\mu= 1e-3$.
We start with a constant initial phase field using $\beta=0.1$. 
The inflow is set to $x\equiv 1$ and we use the parameters $m_i=0.5$, $l_i=0.1$,
$h_i=1$ in \eqref{eq:num:parabolicBoundary} both for the $x$ and $y$ direction
of the boundary velocity field, resulting in an inflow pointing $45\degree$ upwards. 
We set the outflow to $y\equiv 0$ and  consider two scenarios. 
For the first scenario we use the values $m_1=0.8$, $l_1=0.2$ and $h_1=0.5$ in
\eqref{eq:num:parabolicBoundary}, and for the second example we use $m_2=0.3$,
$l_2=0.2$ and $h_2=0.5$.

In Figure \ref{fig:num:FG:results} we show our numerical finding. 
We obtain a straight and wide pipe that directly connects inflow and outflow boundary. This
corresponds to our optimization aim, i.e. minimizing the dissipative power. 
Similar trends for the optimized shape of the embouchure were also observed by the group in
Dresden.

\begin{figure}
  \centering
  \epsfig{file=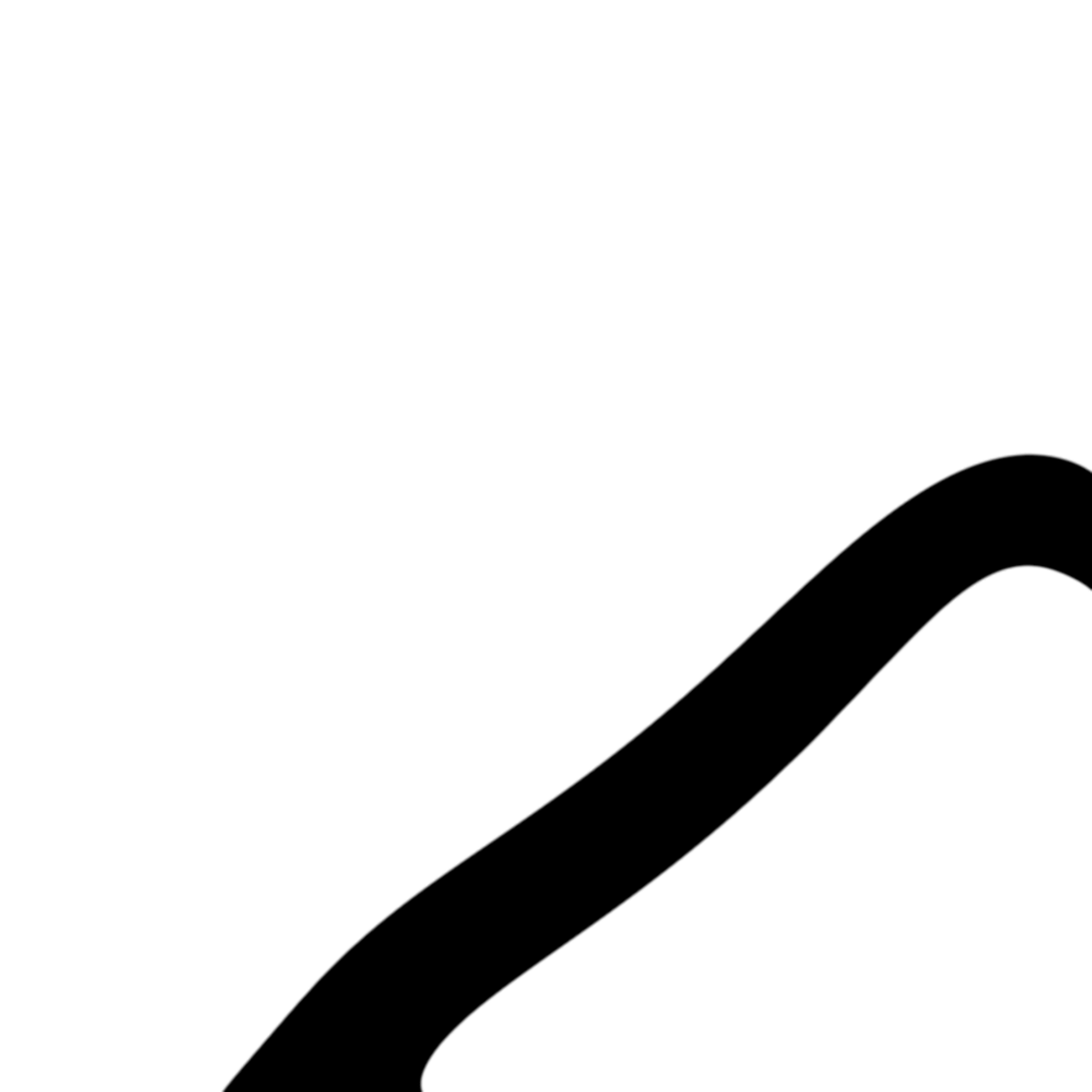,width=0.4\textwidth}
  \hspace{2cm}
  \epsfig{file=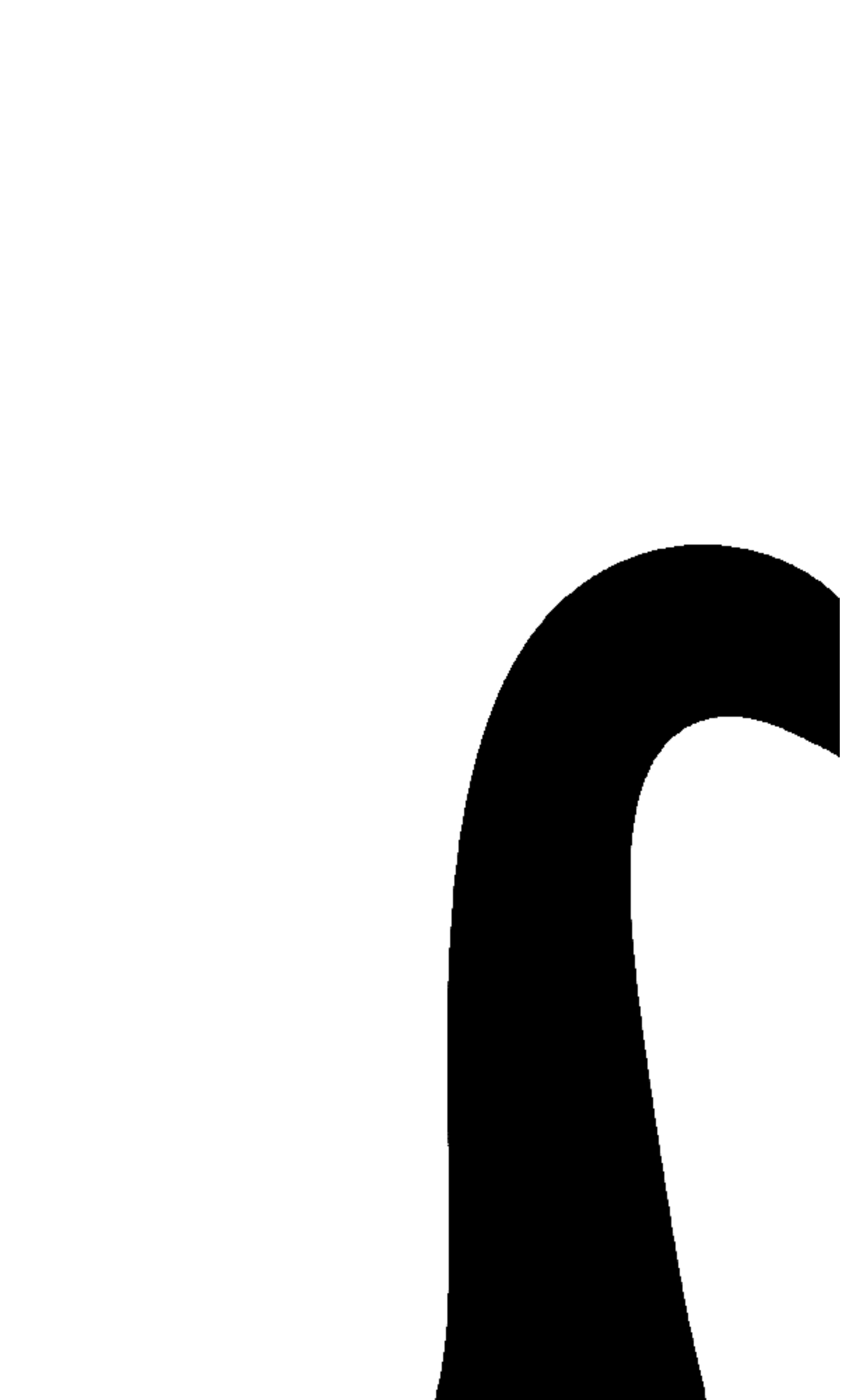,width=0.24\textwidth}  
  \caption{Optimized shapes for the bassoon example for
  $\mu=1000^{-1}$. First scenario on left side, second scenario on right side.
   The inflow is on the right side.}
  \label{fig:num:FG:results}
\end{figure}

\bibliographystyle{plain}
\bibliography{literature}

\begin{thebibliography}{10}

\bibitem{AinsworthOden_Aposteriori}
M.~Ainsworth and J.~T. Oden.
\newblock {\em A Posteriori Error Estimation in Finite Element Analysis}.
\newblock Wiley, September 2000.

\bibitem{ambrosioButtazzo}
L.~Ambrosio and G.~Buttazzo.
\newblock {An optimal design problem with perimeter penalization}.
\newblock {\em Calc. Var. Partial Differential Equations}, 1(1):55--69, 1993.

\bibitem{haberjog}
M.P. Bends{\o}e, R.B. Haber, and C.S. Jog.
\newblock {A new approach to variable-topology shape design using a constraint
  on perimeter}.
\newblock {\em Struct. Multidiscip. Optim.}, 11(1-2):1--12, 1996.

\bibitem{Benzi_numericalSaddlePoint}
M.~Benzi, G.H. Golub, and J.~Liesen.
\newblock Numerical solution of saddle point problems.
\newblock {\em Acta Numer.}, 14:1--137, 2005.

\bibitem{relatingphasefield}
L.~Blank, H.~Farshbaf-Shaker, H.~Garcke, and V.~Styles.
\newblock Relating phase field and sharp interface approaches to structural
  topology optimization.
\newblock {\em to appear in ESAIM: COCV}, 2014.

\bibitem{borrvall}
T.~Borrvall and J.~Petersson.
\newblock {Topology optimization of fluids in Stokes flow}.
\newblock {\em Internat. J. Numer. Methods Fluids}, 41(1):77--107, 2003.

\bibitem{BoschStollBenner_fastSolutionCH}
J.~Bosch, M.~Stoll, and P.~Benner.
\newblock {Fast solution of Cahn-Hilliard Variational Inequalities using
  Implicit Time Discretization and Finite Elements}.
\newblock {\em J. Comput. Phys.}, 2014.

\bibitem{bourdin_chambolle}
B.~Bourdin and A.~Chambolle.
\newblock {Design-dependent loads in topology optimization}.
\newblock {\em ESAIM Control Optim. Calc. Var.}, 9:19--48, 8 2003.

\bibitem{ulbrichulbrich}
C.~Brandenburg, F.~Lindemann, M.~Ulbrich, and S.~Ulbrich.
\newblock {A Continuous Adjoint Approach to Shape Optimization for Navier
  Stokes Flow}.
\newblock In K.~Kunisch, J.~Sprekels, G.~Leugering, and F.~Tr\"oltzsch,
  editors, {\em {Optimal Control of Coupled Systems of Partial Differential
  Equations}}, volume 158 of {\em Internat. Ser. Numer. Math.}, pages 35--56.
  Birkh\"auser, 2009.

\bibitem{BrandenburgLindemannUlbrichUlbrich_advancedNumMeth_DesignNSflow}
C.~Brandenburg, F.~Lindemann, M.~Ulbrich, and S.~Ulbrich.
\newblock {Advanced Numerical Methods for PDE Constrained Optimization with
  Application to Optimal Design in Navier Stokes Flow}.
\newblock In G.~Leugering, S.~Engell, A.~Griewank, M.~Hinze, R.~Rannacher,
  V.~Schulz, M.~Ulbrich, and S.~Ulbrich, editors, {\em Constrained Optimization
  and Optimal Control for Partial Differential Equations}, pages 257--275.
  Birkh\"auser, 2012.

\bibitem{burger}
M.~Burger and R.~Stainko.
\newblock Phase-field relaxation of topology optimization with local stress
  constraints.
\newblock {\em SIAM J. Control Optim.}, 45:1447--1466, 2006.

\bibitem{Carstensen_QuasiInterpolation}
C.~Carstensen.
\newblock {Quasi-interpolation and a-posteriori error analysis in finite
  element methods}.
\newblock {\em ESAIM Math. Model. Numer. Anal.}, 33(6):1187--1202, 1999.

\bibitem{CarstensenVerfuerth_EdgeResidualDominate}
C.~Carstensen and R.~Verf\"urth.
\newblock {Edge Residuals Dominate A Posteriori Error Estimates for Low Order
  Finite Element Methods}.
\newblock {\em SIAM J. Numer. Anal.}, 36(5):1571--1587, 1999.

\bibitem{UMFPACK}
T.~A. Davis.
\newblock Algorithm 832: Umfpack v4.3 - an unsymmetric-pattern multifrontal
  method.
\newblock {\em ACM Trans. Math. Software}, 30(2):196--199, 2004.

\bibitem{Doerfler}
W.~D{\"o}rfler.
\newblock A convergent adaptive algorithm for {Poisson's} equation.
\newblock {\em SIAM J. Numer. Anal.}, 33(3):1106--1124, 1996.

\bibitem{evans_gariepy}
L.C. Evans and R.F. Gariepy.
\newblock {\em {Measure Theory and Fine Properties of Functions}}.
\newblock Mathematical Chemistry Series. CRC PressINC, 1992.

\bibitem{evgrafov}
A.~Evgrafov.
\newblock {The Limits of Porous Materials in the Topology Optimization of
  Stokes Flows}.
\newblock {\em Appl. Math. Optim.}, 52(3):263--277, 2005.

\bibitem{evgrafov2}
A.~Evgrafov.
\newblock {Topology optimization of slightly compressible fluids}.
\newblock {\em ZAMM Z. Angew. Math. Mech.}, 86(1):46--62, 2006.

\bibitem{eyre_CH_semi_implicite}
D.~J. Eyre.
\newblock Unconditionally gradient stable time marching the {Cahn--Hilliard}
  equation.
\newblock In {\em {Computational and Mathematical Models of Microstructural
  Evolution}}, volume 529 of {\em MRS Proceedings}, 1998.

\bibitem{galdi}
G.P. Galdi.
\newblock {\em {An Introduction to the Mathematical Theory of the Navier-Stokes
  Equations}}.
\newblock Springer, 2011.

\bibitem{GarckeHinzeKahle_CHNS_AGG_linearStableTimeDisc}
H.~Garcke, M.~Hinze, and C.~Kahle.
\newblock {A stable and linear time discretization for a thermodynamically
  consistent model for two-phase incompressible flow}.
\newblock {\em arXiv: 1402.6524}, 2014.

\bibitem{Hansen_Dissertation}
A.~Gersborg-Hansen.
\newblock {\em Topology optimization of flow problems}.
\newblock PhD thesis, Technical University of Denmark, 2007.

\bibitem{GiraultRaviart_FEM_for_NavierStokes}
V.~Girault and P.~A. Raviart.
\newblock {\em Finite Element Methods for Navier--Stokes Equations}, volume~5
  of {\em Springer series in computational mathematics}.
\newblock Springer, 1986.

\bibitem{giusti}
E.~Giusti.
\newblock {\em {Minimal surfaces and functions of bounded variation}}.
\newblock Notes on pure mathematics. Dept. of Pure Mathematics, 1977.

\bibitem{grundmann}
R.~Grundmann.
\newblock {Das Fagott und die Str\"omungsmechanik}.
\newblock {\em Forschung}, 29(1):16--17, 2004.

\bibitem{HansenHaberSigmun_TopoOptChannelFlow}
A.~G. Hansen, O.~Sigmund, and R.~B. Haber.
\newblock Topology optimization of channel flow problems.
\newblock {\em Struct. Multidis. Optim.}, 30:181--192, 2005.

\bibitem{hecht}
C.~Hecht.
\newblock {\em {Shape and topology optimization in fluids using a phase field
  approach and an application in structural optimization}}.
\newblock {Dissertation}, University of Regensburg, 2014.

\bibitem{HintermuellerHinzeKahle_AFEM_for_CHNS}
M.~Hinterm\"uller, M.~Hinze, and C.~Kahle.
\newblock {An adaptive finite element Moreau--Yosida-based solver for a coupled
  Cahn--Hilliard/Navier--Stokes system}.
\newblock {\em J. Comput. Phys.}, 235:810--827, February 2013.

\bibitem{HintermuellerHinzeTber__AFEM_for_CH}
M.~Hinterm\"uller, M.~Hinze, and M.H. Tber.
\newblock {An adaptive finite element Moreau--Yosida-based solver for a
  non-smooth Cahn--Hilliard problem}.
\newblock {\em Optim. Methods Softw.}, 25(4-5):777--811, 2011.

\bibitem{kawohl2000optimal}
B.~Kawohl, A.~Cellina, and A.~Ornelas.
\newblock {\em {Optimal Shape Design: Lectures Given at the Joint
  C.I.M./C.I.M.E. Summer School Held in Troia (Portugal), June 1-6, 1998}}.
\newblock Lecture Notes in Mathematics / C.I.M.E. Foundation Subseries.
  Springer, 2000.

\bibitem{kayLoghinWelford_FpPreconditioner}
D.~Kay, D.~Loghin, and A.~Wathen.
\newblock A preconditioner for the steady state {Navier--Stokes} equations.
\newblock {\em SIAM J. Sci. Comput.}, 24(1):237--256, 2002.

\bibitem{KreisslMaute_fluidTopoOptXFEM}
S.~Kreissl and K.~Maute.
\newblock Levelset based fluid topology optimization using the extended finite
  element method.
\newblock {\em Struct. Multidis. Optim.}, 46:311--326, 2012.

\bibitem{modica}
L.~Modica.
\newblock {The gradient theory of phase transitions and the minimal interface
  criterion}.
\newblock {\em Arch. Ration. Mech. Anal.}, 98(2):123--142, 1987.

\bibitem{pironneau}
B.~Mohammadi and O.~Pironneau.
\newblock {Shape optimization in fluid mechanics}.
\newblock {\em Annu. Rev. Fluid Mech.}, 36:255--279, 2004.

\bibitem{OlesenOkelsBruus_TopoOptSteadyStateNS}
L.~H. Olesen, F.~Okkels, and H.~Bruus.
\newblock A high-level programming-language implementation of topology
  optimization applied to steady state {Navier-Stokes} flow.
\newblock {\em Internat. J. Numer. Methods Engrg.}, 65(7):975--1001, 2006.

\bibitem{Payne_PoincareConstant}
L.~E. Payne and H.~F. Weinberger.
\newblock An optimal {Poincar\'e} inequality for convex domains.
\newblock {\em Arch. Ration. Mech. Anal.}, 5(1):286--292, Januar 1960.

\bibitem{pingen_TopoOpt_Boltzmann}
G.~Pingen, A.~Evgrafov, and K.~Maute.
\newblock {Topology optimization of flow domains using the lattice Boltzmann
  method}.
\newblock {\em Struct. Multidis. Optim.}, 34:507--524, 2007.

\bibitem{Saad_gmres}
Y.~Saad and M.~H. Schultz.
\newblock {GMRES: A generalized minimal residual algorithm for solving
  nonsymmetric linear systems}.
\newblock {\em SIAM Journal on Scientific and Statistical Computing},
  7(3):856--869, July 1986.

\bibitem{Schmidt_shape_derivative_NavierStokes}
S.~Schmidt and V.~Schulz.
\newblock {Shape Derivatives for General Objective Functions and the
  Incompressible Navier--Stokes Equations}.
\newblock {\em Control Cybernet.}, 39(3):677--713, 2010.

\bibitem{sverak}
V.~S\v{v}er\'ak.
\newblock On optimal design.
\newblock {\em J. Maths. Pures Appl.}, 72:537--551, 1993.

\bibitem{Wadell_Circularity}
H.~Wadell.
\newblock {Spericity and Roundness of Rock Particles}.
\newblock {\em The Journal of Geology}, 41(3):310--331, 1933.

\end{thebibliography}

\end{document}